\documentclass[12pt]{amsart}
\pdfoutput=1
\usepackage{microtype}

\usepackage{amssymb}
\usepackage[cmtip,all]{xy}
\usepackage{verbatim}
\usepackage{upref}
\usepackage{hyperref}
\usepackage[usenames,dvipsnames]{color}
\usepackage{url}
\usepackage[normalem]{ulem}

\newtheorem{thm}{Theorem}[section]
\newtheorem*{thm*}{Theorem}
\newtheorem{lem}[thm]{Lemma}
\newtheorem{cor}[thm]{Corollary}
\newtheorem{prop}[thm]{Proposition}

\theoremstyle{definition} 
\newtheorem{defn}[thm]{Definition}

\newtheorem{q}[thm]{Question}
\newtheorem{rem}[thm]{Remark}

\numberwithin{equation}{section}

\allowdisplaybreaks

\newcommand{\secref}[1]{Section~\textup{\ref{#1}}}

\newcommand{\thmref}[1]{Theorem~\textup{\ref{#1}}}
\newcommand{\corref}[1]{Corollary~\textup{\ref{#1}}}
\newcommand{\lemref}[1]{Lemma~\textup{\ref{#1}}}
\newcommand{\propref}[1]{Proposition~\textup{\ref{#1}}}

\newcommand{\defnref}[1]{Definition~\textup{\ref{#1}}}

\newcommand{\KK}{\mathcal K}

\newcommand{\variso}{\overset{\simeq}{\longrightarrow}}

\renewcommand{\bar}{\overline}
\newcommand{\what}{\widehat}
\newcommand{\wilde}{\widetilde}
\newcommand{\inv}{^{-1}}

\newcommand{\smtx}[1]{\left(\begin{smallmatrix} #1\end{smallmatrix}\right)}
\newcommand{\mtx}[1]{\begin{pmatrix} #1 \end{pmatrix}}

\renewcommand{\:}{\colon}
\newcommand{\cotimes}{\mathrel{\sharp}}

\renewcommand{\)}{\textup)}
\newcommand{\rt}{\textup{rt}}

\newcommand{\id}{\text{\textup{id}}}
\newcommand{\triv}{\delta_{\textup{triv}}}

\DeclareMathOperator{\ad}{Ad}
\DeclareMathOperator*{\spn}{span}
\DeclareMathOperator*{\clspn}{\overline{\spn}}

\newcommand{\romannum}{\renewcommand{\labelenumi}{(\roman{enumi})}}

\newcommand{\righttext}[1]{\quad\text{#1 }}

\newcommand{\nd}{_\mathbf{nd}}
\newcommand{\en}{_\mathbf{en}}
\newcommand{\ou}{_\mathbf{ou}}
\newcommand{\du}{_\mathbf{sc}}

\newcommand{\cs}{\mathbf{C}^*}
\newcommand{\csn}{\cs\nd}
\newcommand{\cse}{\cs\en}

\newcommand{\ac}{\mathbf{Ac}}
\newcommand{\acn}{\ac_{\mathbf{nd}}}
\newcommand{\ace}{\ac_{\mathbf{en}}}

\newcommand{\wac}{\rt\text{-}{\ac}}
\newcommand{\wacn}{\wac\nd}
\newcommand{\wace}{\wac\en}
\newcommand{\waco}{\wac\ou}

\newcommand{\eqac}{\rt/\ac}

\newcommand{\eqace}{\eqac\en}
\newcommand{\eqaco}{\eqac\ou}
\newcommand{\eqacd}{\eqac\du}

\newcommand{\co}{\mathbf{Co}}
\newcommand{\con}{\co\nd}
\newcommand{\coe}{\co\en}
\newcommand{\coo}{\co\ou}

\newcommand{\kalg}{\KK/\cs}
\newcommand{\kalgn}{\KK\!\text{ - }\!\csn}
\newcommand{\kalge}{\kalg\en}

\newcommand{\kalgen}{\KK\!\text{ - }\!\cse}

\newcommand{\kco}{\KK/\co}
\newcommand{\kcoe}{\kco\en}

\newcommand{\st}{\textup{St}}

\newcommand{\dst}{\textup{DSt}}
\newcommand{\dstn}{\dst\nd}
\newcommand{\dste}{\dst\en}

\newcommand{\wst}{\wilde{\st}}

\newcommand{\wste}{\wst\en}

\newcommand{\cp}{\textup{CP}}

\DeclareMathOperator{\cpc}{CPC}
\DeclareMathOperator{\cpa}{CPA}

\DeclareMathOperator{\cpce}{\cpc\en}

\DeclareMathOperator{\cpcd}{\cpc\du}
\DeclareMathOperator{\cpae}{\cpa\en}

\newcommand{\coem}{\co\en^m}

\newcommand{\cpcm}{\cpc^m}
\newcommand{\cpcem}{\cpc\en^m}

\newcommand{\wcp}{\wilde{\cp}}
\newcommand{\wcpn}{\wcp\nd}

\newcommand{\fix}{\textup{Fix}}
\newcommand{\fixn}{\fix\nd}
\newcommand{\fixe}{\fix\en}

\newcommand{\ma}{\textup{Max}}

\DeclareMathOperator{\nor}{Nor}

\begin{document}
\title{Dualities for maximal coactions}
\author[Kaliszewski]{S.~Kaliszewski}
\address{School of Mathematical and Statistical Sciences
\\Arizona State University
\\Tempe, Arizona 85287}
\email{kaliszewski@asu.edu}
\author[Omland]{Tron Omland}
\address{School of Mathematical and Statistical Sciences
\\Arizona State University
\\Tempe, Arizona 85287}
\email{omland@asu.edu}
\author[Quigg]{John Quigg}
\address{School of Mathematical and Statistical Sciences
\\Arizona State University
\\Tempe, Arizona 85287}
\email{quigg@asu.edu}

\date{March 4, 2016}

\subjclass[2010]{Primary 46L55; Secondary 46M15}
\keywords{action, coaction, crossed-product duality, category equivalence, $C^*$-correspondence, exterior equivalence, outer conjugacy}

\begin{abstract}
We present a new construction of crossed-product duality for maximal coactions that uses Fischer's work on maximalizations.
Given a group $G$ and a coaction $(A,\delta)$ we define a generalized fixed-point algebra as a certain subalgebra of $M(A\rtimes_\delta G \rtimes_{\what\delta} G)$,
and recover the coaction via this double crossed product.
Our goal is to formulate this duality in a category-theoretic context, and one advantage of our construction is
that it breaks down into parts that are easy to handle in this regard.
We first explain this for the category of nondegenerate *-homomorphisms, and then analogously
for the category of $C^*$-correspondences.
Also, we outline partial results for the ``outer'' category, studied previously by the authors.
\end{abstract}

\maketitle

\section{Introduction}\label{intro}

One of the fundamental constructions in the area of $C^*$-dynamical systems is the crossed product, which is a $C^*$-algebra having the same representation theory as the $C^*$-dynamical system.
\emph{Crossed-product duality for $C^*$-dynamical systems} is the recovery of the system from its crossed product.
This recovery plays a central role in many aspects of $C^*$-dynamical systems,
and consequently has a variety of formulations,
differing primarily in one of two ways:
the nature of the dynamical systems under consideration,
and the sense in which the system is to be recovered.
We consider the $C^*$-dynamical systems to be either actions or coactions of a locally compact group $G$.

\emph{Imai-Takai} (or \emph{Imai-Takai-Takesaki}) \emph{duality}
\cite{it}
recovers an action up to Morita equivalence
from its reduced crossed product.
Dually,
\emph{Katayama duality}
\cite{kat}
recovers a coaction up to Morita equivalence from its crossed product.
More precisely, in both cases one recovers the original algebra up to tensoring with the compact operators on $L^2(G)$, by forming the double crossed product.
On the other hand,
some crossed-product dualities recover the $C^*$-dynamical system up to isomorphism:
for example,
\emph{Landstad duality}
\cite{landstad}
recovers an action up to isomorphism from its reduced crossed product,
and
\cite{kqfulllandstad}
recovers it from the full crossed product.
Dually,
\cite{QuiLandstad}
recovers a coaction up to isomorphism from its crossed product.

In recent years, some of these dualities have been put on a categorical footing,
casting the crossed-product procedure as a functor,
and producing a quasi-inverse.
Categorical versions of Landstad duality (for actions or for coactions)
\cite{clda, cldx}
require \emph{nondegenerate categories}, in which the morphisms are equivariant nondegenerate homomorphisms into multiplier algebras.
In these categorical dualities the quasi-inverse is constructed from
a \emph{generalized fixed-point algebra} inside the multiplier algebra of the crossed product.

In \cite{koqlandstad} we prove
categorical versions of Imai-Takai and Katayama duality,
which
require categories in which the morphisms are (isomorphism classes of)
equivariant $C^*$-correspondences,
sometimes referred to as \emph{enchilada categories}.
In that paper, which is partly expository, we also present formulations of
the \emph{nondegenerate dualities}
of \cite{clda, cldx},
to highlight the parallels among the dualities of \cite{koqlandstad}.
In fact, we formulate the \emph{enchilada dualities} in a manner that is closer to the nondegenerate dualities than to
the original theorems of \cite{it, kat, kqfulllandstad},
by 
combining the techniques of generalized fixed-point algebras and linking algebras.

The main innovation in \cite{koqlandstad} is the introduction of
\emph{outer duality},
where the crossed-product functor gives an equivalence between
a category of actions
in which 
outer conjugacy is added to
the morphisms 
of the nondegenerate category
and
a category of coactions
in which the morphisms are required to respect the generalized fixed-point algebras.
The proof of outer duality for actions in \cite{koqlandstad} depends upon a theorem of Pedersen 
\cite[Theorem~35]{pedersenexterior}
(that we had to extend from abelian to arbitrary groups) characterizing exterior equivalent actions in terms of a special equivariant isomorphism of the crossed products.
However, we do not have a fully functioning version of Pedersen's theorem for coactions, and because of that, we were not able to obtain a complete outer duality for coactions.

The structure of \cite{koqlandstad} has
a section in which the nondegenerate, the enchilada, and the outer dualities for actions are presented in parallel form,
followed by a final section containing
dual versions:
the nondegenerate, the enchilada, and a partial outer duality for normal coactions.

In the current paper we investigate to what extent the three dualities in the final section of \cite{koqlandstad} carry over from normal to maximal coactions.
Since the categories of maximal and normal coactions are equivalent \cite{clda},
it is natural to expect that things should go well.
Indeed, the nondegenerate and the enchilada dualities do carry over to maximal coactions. Again, this means that one can recover a maximal coaction from its crossed product in a categorical framework, both for the nondegenerate and the enchilada categories; the nondegenerate case is well-known, whereas the enchilada case is new; we present both to highlight the parallel.
However, outer duality presents even more difficulties with maximal coactions than with normal ones.

The dualities for maximal coactions require a different construction than the one in \cite{koqlandstad}.
In that paper, the dualities
for crossed products by normal coactions recover the coaction via a normal coaction defined on a generalized fixed-point algebra that is contained in the multiplier algebra of the crossed product.
However, this generalized fixed-point algebra is not appropriate for the recovery of a maximal coaction, because the coaction produced by the construction is normal, and indeed it is not clear how one could construct a naturally occurring faithful copy of the maximal coaction inside the multipliers of the crossed product.
One of the more delicate aspects inherent in the theory of crossed products by a coaction $(A,\delta)$ of $G$
is that the image of $A$ inside the multiplier algebra $M(A\rtimes_\delta G)$ is faithful if and only if the coaction $\delta$ is normal.
Thus, to get a faithful copy of $A$ when $\delta$ is maximal we must look elsewhere.
The
main innovation in the current paper is the construction of a
\emph{maximal generalized fixed-point algebra} inside the multipliers of the full crossed product $A\rtimes_\delta G\rtimes_{\what\delta} G$ by the \emph{dual action} $\what\delta$.
And to avoid confusion we refer to the earlier algebras inside $M(A\rtimes_\delta G)$ as
\emph{normal generalized fixed-point algebras}.

Our approach depends heavily upon Fischer's construction \cite{fischer} of the \emph{maximalization} of a coaction.
Fischer's construction,
which we feel deserves more attention,
is based upon the factorization of a stable $C^*$-algebra $A$ as $B\otimes\KK$ (where $\KK$ denotes the algebra of compact operators on a separable infinite-dimensional Hilbert space), via a process that produces the $B$ as a \emph{relative commutant of $\KK$}.
We study this \emph{destabilization} process in detail in \cite{koqstable}.
Our use of Fischer's construction allows us to factor the maximalization process into three
more elementary steps:
first take the
crossed product by the coaction,
then 
the
crossed product by the dual action --- but perturb the double dual coaction by a cocycle ---
and finally 
take the
relative commutant of a naturally occurring copy of $\KK$.
It is our opinion that
that this
decomposition gives
rise to
an improved
set of tools to handle maximalizations.

For our present purposes Fischer's construction allows us to devise a formula for the quasi-inverse
in the categorical formulation of recovery of a coaction up to isomorphism from its crossed product,
using the maximal generalized fixed-point algebra of the dual action.

We use the same maximal generalized fixed-point algebra to prove enchilada  duality for maximal coactions, via standard linking-algebra techniques.

In the final section we discuss some of the challenges in obtaining an outer duality for maximal coactions.

One of the themes running through \cite{koqlandstad} is \emph{good inversion}. We think of recovering the $C^*$-dynamical system from its crossed product as inverting a process.
More precisely, the crossed-product process gives a $C^*$-algebra --- the crossed product itself --- that is part of a dual $C^*$-dynamical system equipped with some extra information.
Extracting just the crossed product from this extra stuff can be regarded as a forgetful functor, and
we call the inversion \emph{good} if this forgetful functor enjoys a certain lifting property.
In parallel to \cite{koqlandstad},
the inversion is good in the case of
nondegenerate duality for maximal coactions
and not good for enchilada duality.

In \cite{BusEch}, Buss and Echterhoff develop a powerful technique that handles both maximal and normal coactions, and indeed any \emph{exotic} coaction in between, in a unified manner, by inventing a generalization of Rieffel's approach to generalized fixed-point algebras and applying it to the dual action on the crossed product. In particular, given a coaction $(A,\delta)$, the techniques of \cite{BusEch} give both a maximalization and a normalization by finding a *-subalgebra $(A\rtimes_\delta G)^{G,\what\delta}_c$ of $M(A\rtimes_\delta G)$ and completing in suitable norms.
This is quite distinct from the technique we employ in this paper, where the maximalization is constructed as a subalgebra of $M(A\rtimes_\delta G\rtimes_{\what\delta} G)$.
Both approaches 
should prove useful.

The second author is funded by the Research Council of Norway (Project no.: 240913).

\section{Preliminaries}\label{prelims}

Throughout, $G$ denotes a second countable infinite locally compact group, and $A,B,\dots$ denote $C^*$-algebras.
We impose the countability assumption on $G$ so that $L^2(G)$ is infinite-dimensional and separable.
With some fussiness, we could handle the inseparable case, but we eschew it here because we have no applications of such generality in mind.
We refer to \cite{koqlandstad} for our conventions regarding actions, coactions, $C^*$-correspondences, and cocycles for coactions.
In this paper we will work exclusively with maximal coactions, whereas normal ones figured prominently in \cite{koqlandstad}.

We recall some notation for the convenience of the reader.
The left and right regular representations of $G$ are $\lambda$ and $\rho$, respectively,
the representation of $C_0(G)$ on $L^2(G)$ by multiplication operators is $M$,
and the unitary element $w_G\in M(C_0(G)\otimes C^*(G))$ is the strictly continuous map $w_G\:G\to M(C^*(G))$ given by the canonical embedding of $G$.
The action of $G$ on $C_0(G)$ by right translation is $\rt$.
The crossed product of an action $(A,\alpha)$ of $G$ is $A\rtimes_\alpha G$,
with universal covariant homomorphism $(i_A,i_G)\:(A,G)\to M(A\rtimes_\alpha G)$,
and we use superscripts $i_A^\alpha,i_G^\alpha$ if confusion is likely.
Recall that a coaction is a nondegenerate faithful homomorphism
$\delta\:A\to M(A\otimes C^*(G))$
such that
$(\delta\otimes\id)\circ\delta=(\id\otimes\delta_G)\circ\delta$
and
$\clspn\{\delta(A)(1\otimes C^*(G))\}=A\otimes C^*(G)$,
where
$\delta_G$ denotes the canonical 
coaction
on $C^*(G)$ given by the integrated form of $\delta_G(s)=s\otimes s$ for $s\in G$.
For a coaction $(A,\delta)$, with crossed product $A\rtimes_\delta G$, we write $(j_A,j_G)$ for the universal covariant homomorphism,
and again we use superscripts $j_A^\delta,j_G^\delta$ if confusion is likely.
If $j_A$ is injective, the coaction $\delta$ is called \emph{normal}.
A \emph{normalization} of a coaction $(A,\delta)$ consists of a normal coaction $(B,\epsilon)$ and a surjective $\delta-\epsilon$ equivariant homomorphism $\phi\:A\to B$
such that the crossed product
$\phi\rtimes G\:A\rtimes_\delta G\to B\rtimes_\epsilon G$
is an isomorphism.
Sometimes the coaction $(B,\epsilon)$ itself is referred to as a normalization.
For an action $(A,\alpha)$,
there is a \emph{dual coaction} $\what\alpha$ of $G$ on $A\rtimes_\alpha G$, and
the nondegenerate homomorphism $i_G\:C^*(G)\to M(A\rtimes_\alpha G)$ is $\delta_G-\what\alpha$ equivariant.
For a coaction $(A,\delta)$,
there is a \emph{dual action} of $G$ on $A\rtimes_\delta G$, and the
nondegenerate homomorphism $j_G\:C_0(G)\to M(A\rtimes_\delta G)$ is equivariant for
$\rt$
and the dual action $\what\delta$ on $A\rtimes_\delta G$.
We
write $\KK=\KK(L^2(G))$, and
identify $C_0(G)\rtimes_{\rt} G=\KK$ and $\what\rt=\ad\rho$.
The regular covariant representation of a coaction $(A,\delta)$ on the Hilbert $A$-module $A\otimes L^2(G)$ is the pair $((\id\otimes\lambda)\circ\delta,1\otimes M)$.
The \emph{canonical surjection}
\[
\Phi_A\:A\rtimes_\delta G\rtimes_{\what\delta} G\to A\otimes \KK
\]
is the integrated form of the covariant homomorphism
\[
\bigl((\id\otimes\lambda)\circ\delta\times (1\otimes M),1\otimes\rho\bigr).
\]
If $\Phi_A$ is injective, the coaction $\delta$ is called \emph{maximal}.
A \emph{maximalization} of a coaction $(A,\delta)$ consists of a maximal coaction $(B,\epsilon)$ and a surjective $\epsilon-\delta$ equivariant homomorphism $\phi\:B\to A$
such that the crossed product
$\phi\rtimes G\:B\rtimes_\epsilon G\to A\rtimes_\delta G$
is an isomorphism.
Sometimes the coaction $(B,\epsilon)$ itself is referred to as a maximalization.

We recall a few facts
from \cite{fischer, maximal, qrtwisted, lprs}
concerning cocycles for coactions.
If $(A,\delta)$ is a coaction and $U\in M(A\otimes C^*(G))$ is a $\delta$-cocycle,
there are a \emph{perturbed coaction} $\ad U\circ\delta$ and
an isomorphism
\begin{equation}\label{Phi}
\Omega_U\:A\rtimes_{\ad U\circ\delta} G\variso A\rtimes_\delta G
\end{equation}
given by
$\Omega_U=j_A^\delta\times \mu$,
where
$\mu\:C_0(G)\to M(A\rtimes_\delta G)$ is the nondegenerate homomorphism determined by the unitary element
\[
(\mu\otimes\id)(w_G)=(j_A^\delta\otimes\id)(U)(j_G^\delta\otimes\id)(w_G)
\]
of $M\bigl((A\rtimes_\delta G)\otimes C^*(G)\bigr)$,
and moreover $\Omega_U$ is $\what{\ad U\circ\delta}-\what\delta$ equivariant.
For reduced coactions, this result is
\cite[Proposition~2.8]{qrtwisted}
(based upon the original version 
\cite[Theorem~2.9]{lprs},
which did not include the dual action).
Applying the equivalence between reduced and normal coactions
\cite[Proposition~3.1 and Theorem~3.4]{boiler},
the result carries over to normal coactions
(see
\cite[Proposition~3.6]{koqlandstad} for the statement).
Then, for any coaction $(A,\delta)$ with normalization $(A^n,\delta^n)$, the result in full generality, as stated above, follows by applying the equivariant isomorphism $A\rtimes_\delta G\simeq A^n\rtimes_{\delta^n} G$ from \cite[Proposition~2.6]{fullred} (and note that, although the dual action is not explicitly mentioned there, the equivariance of the isomorphism is obvious).
If we have another coaction $(B,\epsilon)$ and a nondegenerate $\delta-\epsilon$ equivariant homomorphism $\phi\:A\to M(B)$,
then $(\phi\otimes\id)(U)$ is an $\epsilon$-cocycle,
and $\phi$ is also
\[
(\ad U\circ\delta)-\bigl(\ad (\phi\otimes\id)(U)\circ\epsilon\bigr)
\]
equivariant.

The homomorphism
\[
\delta\otimes_*\id:=(\id\otimes\Sigma)\circ (\delta\otimes\id)\:
A\otimes \KK\to M(A\otimes \KK\otimes C^*(G)),
\]
where
$\Sigma\:\KK\otimes C^*(G)\to C^*(G)\otimes \KK$ is the flip isomorphism,
is a coaction.

The unitary
$w_G\in M(C_0(G)\otimes C^*(G))$
is a cocycle for the trivial coaction $\triv$ of $G$ on $C_0(G)$.
The multiplication representation $M$ of $C_0(G)$ on $L^2(G)$ may be regarded as a 
$\triv-\what\rt$ equivariant nondegenerate homomorphism
from $C_0(G)$ to the multiplier algebra $M(\KK)$,
so
$(M\otimes\id)(w_G)$ is a $\what\rt$-cocycle.
The perturbed coaction
$\ad (M\otimes\id)(w_G)\circ\what\rt$ is trivial,
as one checks with a routine calculation using
covariance of the pair $(M,\rho)$ and the identity
$w_G(1\otimes s)=(\rt_s\otimes\id)(w_G)$.
In other words, the coaction $\what\rt$ on $\KK$ is the inner coaction implemented by the 
unitary
\[
W:=(M\otimes\id)(w_G^*)\in M(\KK\otimes C^*(G)).
\]

It follows that if $(A,\delta)$ is a coaction then
$1_{M(A)}\otimes W$ is a $(\delta\otimes_*\id)$-cocycle.
Denote the perturbed coaction by
\[
\delta\otimes_* W=\ad (1\otimes W)\circ (\delta\otimes_*\id).
\]
Then the canonical surjection
$\Phi_A\:A\rtimes_\delta G\rtimes_{\what\delta} G\to A\otimes\KK$
is $\what{\what\delta}-(\delta\otimes_* W)$ equivariant.

\propref{Gamma} below is a reformulation of \cite[Section~1, particularly Lemma~1.16]{fischer}, and we state it formally for convenient reference.
First we introduce the pieces that will combine to make the statement of the proposition:

Consider the following diagram~\hypertarget{big}{\textup{(2.2)}}:
\begin{equation*}
\xymatrix@C-20pt{
(A\otimes\KK)\rtimes_{\delta\otimes_* W} G
\ar@{-->}[rr]^-{\Gamma_A}_-\simeq
\ar[dd]_{(\phi\otimes\id)\rtimes G}
\ar[dr]_{\Omega_{1\otimes W}}^\simeq
&&(A\rtimes_\delta G)\otimes\KK
\ar[dd]^{(\phi\rtimes G)\otimes\id}
\\
&(A\otimes\KK)\rtimes_{\delta\otimes_*\id} G
\ar'[d]^{(\phi\otimes\id)\rtimes G}[dd]
\ar[ur]^\simeq
\\
M((B\otimes\KK)\rtimes_{\epsilon\otimes_* W} G)
\ar@{-->}[rr]^(.4){\Gamma_B}_(.4)\simeq
\ar[dr]_{\Omega_{1\otimes W}}^\simeq
&&M((B\rtimes_\epsilon G)\otimes\KK)
\\
&M((B\otimes \KK)\rtimes_{\epsilon\otimes_*\id} G).
\ar[ur]^\simeq
}
\end{equation*}

The upper southeast arrow $\Omega$ is the isomorphism associated to the $(\delta\otimes_*\id)$-cocycle $1_{M(A)}\otimes W$,
and is
\[
\what{\delta\otimes_* W}-\what{\delta\otimes_*\id}
\]
equivariant,
and similarly the lower southeast arrow is $\what{\epsilon\otimes_* W}-\what{\epsilon\otimes_*\id}$ equivariant.

The upper northeast arrow is the isomorphism
\begin{align*}
&(j_A^\delta\otimes\id_\KK)\times (j_G^\delta\otimes 1_{M(\KK)}),
\end{align*}
and is
\[
\what{\delta\otimes_*\id}-(\what\delta\otimes\id)
\]
equivariant,
and similarly the lower northeast arrow is $\what{\epsilon\otimes_*\id}-(\what\epsilon\otimes\id)$ equivariant.

The left-hand vertical arrow is the crossed product of the
\[
(\delta\otimes_* W)-(\epsilon\otimes_* W)
\]
equivariant homomorphism $\phi\otimes\id_\KK$,
and similarly the middle vertical arrow is the crossed product of $\phi\otimes\id$, but now regarded as being $(\delta\otimes_*\id)-(\epsilon\otimes_*\id)$ equivariant.

The right-hand vertical arrow is the tensor product with $\id_\KK$ of the crossed product $\phi\rtimes G$ of $\phi$,
and is
\[
(\what\delta\otimes\id)-(\what\epsilon\otimes\id)
\]
equivariant.

The upper horizontal arrow $\Gamma_A$ is defined so that the top triangle commutes,
and is
\[
\what{\delta\otimes_* W}-(\what\delta\otimes\id)
\]
equivariant,
and similarly the lower horizontal arrow $\Gamma_B$ is $\what{\epsilon\otimes_* W}-(\what\epsilon\otimes\id)$ equivariant.

\begin{prop}[Fischer]\label{Gamma}
The diagram~\hyperlink{big}{\textup{(2.2)}} commutes.
\end{prop}

\begin{proof}
This follows from naturality of the southeast and northeast isomorphisms.
\end{proof}

\subsection*{Relative commutants}
If $\iota\:\KK\to M(A)$ is a nondegenerate homomorphism,
the \emph{$A$-relative commutant of $\KK$} is
\[
C(A,\iota):=\{m\in M(A):m\iota(k)=\iota(k)m\in A\text{ for all }k\in\KK\}.
\]
The canonical isomorphism $\theta_A\:C(A,\iota)\otimes\KK\variso A$ is determined by $\theta(m\otimes k)=m\iota(k)$
(see \cite[Remark~3.1]{fischer} and \cite[Proposition~3.4]{koqstable}).

As Fischer observes, $C(A,\iota)$ can be characterized as the unique closed subset $Z$ of $M(A)$ 
that commutes elementwise with $\iota(\KK)$ 
and 
satisfies
$\clspn \{Z\iota(\KK)\}=A$, since trivially such a $Z$ is contained in $C(A,\iota)$ while on the other hand the isomorphism $\theta_A$ shows that $Z$ cannot be a proper subset of $C(A,\iota)$.

Also, $M(C(A,\iota))$ can be characterized as the set of all elements of $M(A)$ commuting with the image of $\iota$, since this set is the idealizer of the nondegenerate subalgebra $C(A,\iota)$ of $M(A)$
(see \cite[Lemma~3.6]{koqstable}, alternatively \cite[Remark~3.1]{fischer} again).

\subsection*{$C^*$-correspondences}
We refer to \cite{enchilada, Lance} for the basic definitions and facts we will need regarding $C^*$-correspondences (but note that in \cite{enchilada} correspondences are called \emph{right-Hilbert bimodules}).
An $A-B$ correspondence $X$ is \emph{nondegenerate} if $AX=X$, and we always assume that our correspondences are nondegenerate.
If $Y$ is a $C-D$ correspondence and $\pi\:A\to M(C)$ and $\rho\:B\to M(D)$ are nondegenerate homomorphisms,
a $\pi-\rho$ compatible correspondence homomorphism $\psi\:X\to M(Y)$ is \emph{nondegenerate} if $\clspn\{\psi(X)D\}=Y$.
If $\delta$ and $\epsilon$ are coactions of $G$ on $A$ and $B$, respectively,
a \emph{coaction} of $G$ on $X$ is a nondegenerate $\delta-\epsilon$ correspondence homomorphism
$\zeta\:X\to M(X\otimes C^*(G))$ such that
$\clspn\{(1\otimes C^*(G))\zeta(X)\}=X\otimes C^*(G)$
and $(\zeta\otimes\id)\circ\zeta=(\id\otimes\delta_G)\circ\zeta$
(and then $\clspn\{\zeta(X)(1\otimes C^*(G))\}=X\otimes C^*(G)$ automatically holds).
If $X$ is an $A-B$ correspondence and $Y$ is a $B-C$ correspondence then
for any $C^*$-algebra $D$ there is an isomorphism
\[
\Theta\:(X\otimes D)\otimes_{B\otimes D} (Y\otimes D)\variso (X\otimes_B Y)\otimes D
\]
of $(A\otimes D)-(C\otimes D)$ correspondences,
given by
\[
\Theta\bigl((x\otimes d)\otimes (y\otimes d')\bigr)=(x\otimes y)\otimes dd'\righttext{for}x\in X,y\in Y,d,d'\in D,
\]
and we use this to form \emph{balanced tensor products of coactions} as follows:
given a $\delta-\epsilon$ compatible coaction $\zeta$ on $X$ and an $\epsilon-\varphi$ compatible coaction $\eta$ on $Y$,
we get a $\delta-\varphi$ compatible coaction
\[
\zeta\cotimes_B \eta:=\Theta\circ (\zeta\otimes_B\eta)
\]
on $X\otimes_B Y$.

\section{Nondegenerate categories and functors}\label{categories}

\subsection*{Categories}

We will recall some categories from \cite{koqlandstad} and \cite{koqstable} (and \cite{enchilada}), and we will introduce a few more.
Everything will be based upon
the \emph{nondegenerate category $\cs$ of $C^*$-algebras}, where a morphism $\phi\:A\to B$ is a nondegenerate homomorphism $\phi\:A\to M(B)$.
By nondegeneracy $\phi$ has a canonical extension
$\bar\phi\:M(A)\to M(B)$,
although
frequently we abuse notation by calling the extension $\phi$ also.
One notable exception to this abusive convention occurs in \secref{outer duality}, where we have to pay closer attention to the extensions $\bar\phi$.

The \emph{nondegenerate category $\co$ of coactions}
has coactions $(A,\delta)$ of $G$ as objects,
and a morphism $\phi\:(A,\delta)\to (B,\epsilon)$ in $\co$
is a morphism $\phi\:A\to B$ in $\cs$ that is \emph{$\delta-\epsilon$ equivariant}.
We write $\co^m$ for the full subcategory of $\co$ whose objects are the maximal coactions,
and $\co^n$ for the full subcategory of normal coactions.

The \emph{nondegenerate category $\ac$ of actions}
has actions $(A,\alpha)$ of $G$ as objects,
and a morphism $\phi\:(A,\alpha)\to (B,\beta)$ in $\ac$
is a morphism $\phi\:A\to B$ in $\cs$ that is \emph{$\alpha-\beta$ equivariant}.

We denote the coslice category (a special case of a \emph{comma category})
of actions under $(C_0(G),\rt)$ by
$\eqac$,
and we denote
an object by $(A,\alpha,\mu)$,
where 
$\mu\:(C_0(G),\rt)\to (A,\alpha)$ is a morphism in $\ac$.
Thus, a morphism $\phi\:(A,\alpha,\mu)\to (B,\beta,\nu)$ in $\eqac$ is a
commuting triangle
\[
\xymatrix{
&(C_0(G),\rt) \ar[dl]_\mu \ar[dr]^\nu
\\
(A,\alpha) \ar[rr]^\phi
&&(B,\beta)
}
\]
in $\ac$.
We call an object in $\eqac$ an \emph{equivariant action} \cite[Definition~2.8]{koqlandstad}.

We denote the coslice category of $C^*$-algebras under $\KK$ by
$\kalg$,
and we denote an object by $(A,\iota)$, where $\iota\:\KK\to A$ is a morphism in $\cs$.
Thus, a morphism $\phi\:(A,\iota)\to (B,\jmath)$ in $\kalg$ is a 
commuting triangle
\[
\xymatrix{
&\KK \ar[dl]_{\iota} \ar[dr]^{\jmath}
\\
A \ar[rr]^\phi
&&B
}
\]
in $\cs$.
We call an object in $\kalg$ a \emph{$\KK$-algebra}.

We denote the coslice category of coactions under the trivial coaction $(\KK,\triv)$ by
$\kco$,
and we denote an object by $(A,\delta,\iota)$, where $\iota\:(\KK,\triv)\to (A,\delta)$ is a morphism in $\co$,
i.e., $\delta\circ\iota=\iota\otimes 1$.
Thus, a morphism $\phi\:(A,\delta,\iota)\to (B,\epsilon,\jmath)$ in $\kco$ is a 
commuting triangle
\[
\xymatrix{
&(\KK,\triv) \ar[dl]_{\iota} \ar[dr]^{\jmath}
\\
(A,\delta) \ar[rr]^\phi
&&(B,\epsilon)
}
\]
in $\co$,
i.e.,
$\phi\:(A,\delta)\to (B,\epsilon)$ is a morphism in $\co$
and $\phi\:(A,\iota)\to (B,\jmath)$ is a morphism in $\kalg$.
We call an object in $\kco$ a \emph{$\KK$-coaction}.

In \cite[Subsection~5.1]{koqlandstad} we used the notations
$\csn$ for $\cs$
and
$\acn$ for $\ac$.
In \cite[Subsection~6.1]{koqlandstad} we wrote
$\con$ for the full subcategory of the present category $\co$ whose objects are the normal coactions,
and
$\wacn$ for $\eqac$ (and called it the \emph{nondegenerate equivariant category of actions}).
In \cite[Definition~4.1]{koqstable} we used the notation $\kalgn$ for $\kalg$.

\subsection*{Functors}

First, we define a functor
\[
\cpc\:\co\to \eqac
\]
on objects by
\[
\cpc(A,\delta)=(A\rtimes_\delta G,\what\delta,j_G),
\]
and
on morphisms
as follows: if
$\phi\colon (A,\delta)\to (B,\epsilon)$
is a morphism in $\co$
then
\[
\cpc(\phi)\:
(A\rtimes_\delta G,\what\delta,j_G^\delta)\to (B\rtimes_\epsilon G,\what\epsilon,j_G^\epsilon)
\]
is the morphism in $\eqac$ given by 
$\cpc(\phi)=\phi\rtimes G$.

We want to use equivariant actions to generate $\KK$-coactions.
Given an equivariant action $(A,\alpha,\mu)$, we can form the dual coaction $(A\rtimes_\alpha G,\what\alpha)$,
and we can also form the $\KK$-algebra $(A\rtimes_\alpha G,\mu\rtimes G)$.
But there is a subtlety: it is easy to see that $\what\alpha$ is not trivial on the image of $\mu\rtimes G$.
We need to perturb the coaction by a cocycle, and we adapt a technique from \cite[Lemma~3.6]{maximal}:

\begin{lem}\label{cocycle}
Let $(A,\alpha,\mu)$ be an equivariant action.
Define
\[
V_A=\bigl((i_A\circ\mu)\otimes\id\bigr)(w_G)\in M\bigl((A\rtimes_\alpha G)\otimes C^*(G)\bigr).
\]
Then $V_A$
is an $\what\alpha$-cocycle.
Denote the perturbed coaction by
\[
\wilde\alpha:=\ad V_A\circ \what\alpha.
\]
Then $(A\rtimes_\alpha G,\wilde\alpha,\mu\rtimes G)$ is a maximal $\KK$-coaction.
\end{lem}

\begin{proof}
We could apply the proof from \cite[Lemmas~3.6--3.7]{maximal} to show that $V_A$ is an $\what\alpha$-cocycle and $\wilde\alpha$ is trivial on the image of $\mu\rtimes G$, because by Landstad duality \cite[Theorem~3.3]{QuiLandstad} the equivariant action $(A,\alpha,\mu)$ is isomorphic to one of the form $(B\rtimes_\delta G,\what\delta,j_G)$ for a coaction $(B,\delta)$.
We prefer to use the technology of cocycles more directly, however:
$\mu\:(C_0(G),\rt)\to (A,\alpha)$ is a morphism in $\ac$,
so
\[
\mu\rtimes G\:(\KK,\what\rt)\to (A\rtimes_\alpha G,\what\alpha)
\]
is a morphism in $\co$, where we recall that we are abbreviating $\KK=\KK(L^2(G))$.
The unitary
\[
(M\otimes\id)(w_G)
\]
is a $\what\rt$-cocycle,
so
\begin{align*}
\bigl((\mu\rtimes G)\otimes\id\bigr)\bigl((M\otimes\id)(w_G)\bigr)
&=\bigl((i_A\circ\mu)\otimes\id\bigr)(w_G)
\end{align*}
is an $\what\alpha$-cocycle.

Further, recall that the perturbed coaction
$\ad (M\otimes\id)(w_G)\circ\what\rt$ is trivial.
Since $\mu\rtimes G\:\KK\to M(A\rtimes_\alpha G)$ is $\what\rt-\what\alpha$ equivariant,
it follows that $\ad V_A\circ\what\alpha$ is trivial on the image of $\mu\rtimes G$.

Finally,
the dual coaction $\what\alpha$ is maximal by \cite[Proposition~3.4]{maximal}, and hence so is the Morita equivalent coaction $\wilde\alpha$ by \cite[Proposition~3.5]{maximal}.
\end{proof}

Next, we define a functor
\[
\cpa\:\eqac\to \kco
\]
on objects by
\[
\cpa(A,\alpha,\mu)=(A\rtimes_\alpha G,\wilde\alpha,\mu\rtimes G),
\]
and
on morphisms
as follows: if
$\phi\colon (A,\alpha,\mu)\to (B,\beta,\nu)$
is a morphism in $\eqac$
then
\[
\cpa(\phi)\:
(A\rtimes_\alpha G,\wilde\alpha,\mu\rtimes G)\to (B\rtimes_\beta G,\wilde\beta,\nu\rtimes G)
\]
is the morphism in $\kco$ given by 
$\cpa(\phi)=\phi\rtimes G$.

It follows from \cite[Theorem~4.4]{koqstable} (see also \cite[Definition~4.5]{koqstable})
that there is a functor
from $\kalg$ to $\cs$
that takes an object $(A,\iota)$ to the $A$-relative commutant $C(A,\iota)$ of $\KK$
and a morphism $\phi\:(A,\iota)\to (B,\jmath)$ to
\[
C(\phi)=\phi|_{C(A,\iota)},
\]
and
that 
is 
moreover 
a quasi-inverse of the \emph{stabilization functor}
given by $A\mapsto A\otimes\KK$ and $\phi\mapsto \phi\otimes\id_\KK$.

We need an equivariant version of 
this
functor,
and we need to see what to do with coactions.
Fischer remarks in \cite[Remark~3.2]{fischer} that
$\delta$
restricts to a coaction on the relative commutant $C(A,\iota)$.
To keep our treatment of the Fischer construction self-contained, we supply a proof (that is similar to, but not quite the same as, the one in \cite{fischer}):

\begin{lem}[Fischer]\label{coact commutant}
Let $(A,\delta,\iota)$ be a $\KK$-coaction.
Then
$\delta$ restricts to a coaction $C(\delta)$ on $C(A,\iota)$.
Moreover, 
the canonical isomorphism
\[
\theta_A\:C(A,\iota)\otimes\KK\variso A
\]
is $(C(\delta)\otimes_*\id)-\delta$ equivariant.
Finally,
$C(\delta)$ is maximal if and only if $\delta$ is.
\end{lem}

\begin{proof}
For the first part, by \cite[Lemma~1.6 (b)]{fullred}, it suffices to show that
\begin{equation}\label{nondegenerate}
\clspn\bigl\{\delta\bigl(C(A,\iota)\bigr)\bigl(1\otimes C^*(G)\bigr)\bigr\}=C(A,\iota)\otimes C^*(G).
\end{equation}
Let $Z$ denote the left-hand side.
It
suffices to show that 
\begin{enumerate}
\romannum
\item
$Z$ commutes elementwise with $\iota(\KK)\otimes 1$,
and

\item
$\clspn\{Z(\iota(\KK)\otimes 1)\}=A\otimes C^*(G)$.
\end{enumerate}
For (i), if $a\in C(A,\iota)$, $x\in C^*(G)$, and $k\in\KK$, then
\begin{align*}
\delta(a)(1\otimes x)(\iota(k)\otimes 1)
&=\delta(a)(\iota(k)\otimes 1)(1\otimes x)
\\&=\delta(a\iota(k))(1\otimes x)
\\&=\delta(\iota(k)a)(1\otimes x)
\\&=(\iota(k)\otimes 1)\delta(a)(1\otimes x).
\end{align*}
For (ii), the above computation implies that
\begin{align*}
\clspn\{Z(\iota(\KK)\otimes 1)\}
&=\clspn\bigl\{\delta\bigl(C(A,\iota)\iota(\KK)\bigr)\bigl(1\otimes C^*(G)\bigr)\bigr\}
\\&=\clspn\{\delta(A)(1\otimes C^*(G))\}
\\&=A\otimes C^*(G).
\end{align*}

One readily checks the $(C(\delta)\otimes_*\id)-\delta$ equivariance
on the generators $a\otimes 1$ and $1\otimes k$ for $a\in C(A,\iota)$ and $k\in\KK$.
The last statement now follows since the coaction $C(\delta)\otimes_*\id$ is Morita equivalent to $C(\delta)$.
\end{proof}

\lemref{coact commutant} allows us to define a functor
\[
\dst\:\kco\to \co
\]
on objects by
\[
\dst(A,\delta,\iota)=\bigl(C(A,\iota),C(\delta)\bigr),
\]
and on morphisms as follows:
if $\phi\:(A,\delta,\iota)\to (B,\epsilon,\jmath)$ is a morphism in $\kco$ then
\[
\dst(\phi)\:\bigl(C(A,\iota),C(\delta)\bigr)\to \bigl(C(B,\jmath),C(\epsilon)\bigr)
\]
is the morphism in $\co$ given by $\dst(\phi)=C(\phi)$,
where we note that the morphism $C(\phi)\:C(A,\iota)\to C(B,\jmath)$ in $\cs$ is $C(\delta)-C(\epsilon)$ equivariant since $\phi$ is $\delta-\epsilon$ equivariant.

In \cite[Definition~4.5]{koqstable} we used the notation $\dstn$ for the 
non-equivariant destabilization functor
$(A,\iota)\mapsto C(A,\iota)$,
while in the current paper we
use $\dst$ to denote the equivariant destabilization functor,
and give no name to the non-equivariant version.

\section{The Fischer construction}\label{Fischer construction}

In this section we define a particular maximalization functor (see \defnref{max def} below).
We will also need to refer to a normalization functor.
A small preliminary discussion might help clarify our approach to these two functors.
Both maximalization and normalization satisfy universal properties, and this can be used to work with the functors abstractly:
for example,
once one knows that maximalizations exist,
then one can use the categorical axiom of choice to assume that
a maximalization has been chosen for every coaction,
and then the universal property takes care of the morphisms.
However, we will find it useful to have a specific construction of the maximalization of a coaction, for example when working with linking algebras of Hilbert bimodules.
For this reason we will define the maximalization functor concretely.
Similarly for normalizations,
although here the choice is somewhat more immediate:
we define the normalization $(A^n,\delta^n)$ of a coaction $(A,\delta)$ by taking $A^n$ as a suitable quotient of $A$.
Then the universal property gives a functor $\nor$ on the nondegenerate category $\co$ of coactions of $G$.

In \cite[Section~3]{fischer} Fischer constructs a maximalization of a coaction $(A,\delta)$.
Actually, Fischer works in the more general context of coactions by Hopf $C^*$-algebras, which occasionally introduces minor complications (such as the existence of the maximalization).
Consequently, the construction simplifies somewhat since we have specialized to coactions of a locally compact group $G$.
Here we present Fischer's construction as a composition of three functors:

\begin{defn}[Fischer]\label{max def}
We define the \emph{maximalization functor $\ma$} by the commutative diagram
\[
\xymatrix@C+30pt{
\co \ar[r]^-{\cpc} \ar[d]_{\ma}
&\eqac \ar[d]^{\cpa}
\\
\co
&\kco \ar[l]^{\dst}.
}
\]
\end{defn}

Thus, given a coaction $(A,\delta)$,
we first form the equivariant action
\[
(A\rtimes_\delta G,\what\delta,j_G),
\]
then the $\KK$-coaction
\[
(A\rtimes_\delta G\rtimes_{\what\delta} G,\wilde{\what\delta},j_G\rtimes G),
\]
where $\wilde{\what\delta}$ is the perturbation of the double-dual coaction $\what{\what\delta}$ from \lemref{cocycle},
and finally the maximal coaction
\begin{align*}
\ma(A,\delta)
&=\dst\bigl((A\rtimes_\delta G\rtimes_{\what\delta} G,\wilde{\what\delta},j_G\rtimes G\bigr)
\\&=\bigl(C(A\rtimes_\delta G\rtimes_{\what\delta} G,j_G\rtimes G),C(\wilde{\what\delta})\bigr).
\end{align*}
We simplify some of the notation:
first, write
\[
\wilde\delta=\wilde{\what\delta}.
\]
Thus 
the second stage in the above process gives the $\KK$-coaction
\[
(A\rtimes_\delta G\rtimes_{\what\delta} G,\wilde\delta,j_G\rtimes G)
\]
and 
the maximalization is
\[
\ma(A,\delta)=\bigl(C(A\rtimes_\delta G\rtimes_{\what\delta} G,j_G\rtimes G),C(\wilde\delta)\bigr).
\]
We will actually use the more customary notation:
\begin{align*}
A^m&=C(A\rtimes_\delta G\rtimes_{\what\delta} G,j_G\rtimes G)
\\
\delta^m&=C(\wilde\delta),
\end{align*}
where the superscript $m$ is intended to remind us of ``maximalization''.

We need a couple of results from \cite{fischer}, which we recall here.
The first is \cite[Theorem~6.4]{fischer}:

\begin{thm}[Fischer]
With the above notation,
$(A^m,\delta^m)$ is a maximalization of $(A,\delta)$.
\end{thm}

We
outline Fischer's argument here:
the commutative diagram
\[
\xymatrix{
(A\rtimes_\delta G\rtimes_{\what\delta} G,\what{\what\delta}) \ar[r]^-\Phi
&(A\otimes\KK,\delta\otimes_* W)
\\
(A^m\otimes\KK,\delta^m\otimes_* W), \ar[u]^{\theta_{A\rtimes_\delta G\rtimes_{\what\delta} G}}_\simeq \ar@{-->}[ur]_{\psi_A\otimes\id}
}
\]
uniquely defines a surjective equivariant homomorphism as the north-east arrow,
which must be of the form $\psi_A\otimes\id$ for a unique surjection $\psi_A$
since the vertical and horizontal homomorphisms 
are morphisms of $\KK$-algebras.
We refer to \secref{prelims} for the notations $\Phi$ and $W$,
as well as $\Gamma$ (see the following diagram), etc.

The crossed product of this diagram fits into a commutative diagram
\[
\xymatrix@C+30pt{
A\rtimes_\delta G\rtimes_{\what\delta} G\rtimes_{\what{\what\delta}} G
\ar[r]^-{\Phi\rtimes G}_-\simeq
&(A\otimes\KK)\rtimes_{\delta\otimes_* W} G \ar[dd]^{\Gamma_A}_\simeq
\\
(A^m\otimes\KK)\rtimes_{\delta^m\otimes_*W} G
\ar[d]_{\Gamma_{A^m}}^\simeq
\ar[u]^{\theta_{A\rtimes_\delta G\rtimes_{\what\delta} G}\rtimes G}_\simeq \ar[ur]_{(\psi_A\otimes\id)\rtimes G}
\\
(A^m\rtimes_{\delta^m} G)\otimes\KK \ar[r]_-{(\psi_A\rtimes G)\otimes\id}
&(A\rtimes_\delta G)\otimes\KK.
}
\]
The vertical maps and the top horizontal map are isomorphisms,
and it follows that the bottom map $(\psi_A\rtimes G)\otimes\id$ is also an isomorphism,
and hence so is $\psi_A\rtimes G$ (because $\KK$ is nuclear,
or alternatively by applying the destabilization functor).
Therefore the map $\psi_A\:(A^m,\delta^m)\to (A,\delta)$ is a maximalization.

The second of Fischer's results that we need is
the \emph{universal property} of maximalizations
\cite[Lemma~6.2]{fischer}:

\begin{lem}[Fischer]
If $\psi_A\:(A^m,\delta^m)\to (A,\delta)$ is any \(abstract\) maximalization,
i.e., $\delta^m$ is maximal, $\psi_A$ is surjective, and $\psi_A\rtimes G$ is an isomorphism,
then
given a morphism $\phi\:(B,\epsilon)\to (A,\delta)$ in $\co$,
where $\epsilon$ is maximal,
there is a unique morphism $\phi'$ giving a commutative completion of the diagram
\[
\xymatrix{
(B,\epsilon) \ar@{-->}[r]^-{\phi'} \ar[dr]_\phi
&(A^m,\delta^m) \ar[d]^{\psi_A}
\\
&(A,\delta).
}
\]
\end{lem}

Fischer's argument
argument consists
of carefully considering the commutative diagram
\[
\xymatrix@C+20pt{
B\rtimes_\epsilon G\rtimes_{\what\epsilon} G \ar[r]^-{\phi\rtimes G\rtimes G}
\ar[d]_{\Phi_B}^\simeq
&A\rtimes_\delta G\rtimes_{\what\delta} G \ar[d]^{\Phi_A}
&A^m\rtimes_{\delta^m} G\rtimes_{\what{\delta^m}} G
\ar[l]_-{\psi_A\rtimes G\rtimes G}^-\simeq \ar[d]^{\Phi_{A^m}}_\simeq
\\
B\otimes\KK \ar[r]_-{\phi\otimes\id}
\ar@{-->}@/_2pc/[rr]_{\Phi_{A^m}\circ (\psi_A\rtimes G\rtimes G)\inv\circ (\phi\rtimes G\rtimes G)\circ \Phi_B\inv}
&A\otimes\KK
&A^m\otimes\KK \ar[l]^-{\psi_A\otimes\id}
}
\]
which we have altered slightly so that it may be viewed as a diagram in the category $\cs$.
Actually, all the maps are equivariant, so it can also be regarded as a diagram in $\co$.
Then because all the maps are morphisms of $\KK$-algebras we get a unique morphism $\phi'$ as required.
The universal property in turn shows that maximalizations are unique up to isomorphism.
In particular, any maximalization $\phi\:(A,\delta)\to (B,\epsilon)$, where $\epsilon$ is itself maximal, must be an isomorphism.

As we discussed at the beginning of this section, now that we have chosen a maximalization of each coaction $(A,\delta)$, the universal property tells what to do to morphisms,
giving us not only a maximalization functor $\ma$,
but also telling us
that the surjections $\psi_A$ give a natural transformation from 
$\ma$
to the identity functor, i.e., if $\phi\:(A,\delta)\to (B,\epsilon)$ is a morphism of coactions, then the diagram
\[
\xymatrix@C+20pt{
(A^m,\delta^m) \ar@{-->}[r]^-{\phi^m}_-{!} \ar[d]_{\psi_A}
&(B^m,\epsilon^m) \ar[d]^{\psi_B}
\\
(A,\delta) \ar[r]_-\phi
&(B,\epsilon)
}
\]
has a unique completion $\phi^m$.

\section{Nondegenerate Landstad duality}\label{nondegenerate duality}

In this section we state the (known) categorical Landstad duality for maximal coactions,
in a form suitable for comparison with the later version for the enchilada categories
(see \secref{enchilada duality}).

First we must cobble together a couple of functors from \secref{categories} to produce a new functor:

\begin{defn}\label{fix functor}
We define a functor $\fix$ by the commutative diagram
\[
\xymatrix{
\eqac \ar[r]^-{\cpa} \ar[dr]_{\fix}
&\kco \ar[d]^{\dst}
\\
&\co.
}
\]
\end{defn}

Thus, given an equivariant action $(A,\alpha,\mu)$,
we first form the $\KK$-coaction
\[
(A\rtimes_\alpha G,\wilde\alpha,\mu\rtimes G),
\]
where $\wilde\alpha$ is the perturbation of the dual action $\what\alpha$ from \lemref{cocycle},
then the
coaction
\[
\bigl(C(A\rtimes_\alpha G,\mu\rtimes G),C(\wilde\alpha)\bigr),
\]
which is maximal
by \lemref{coact commutant}, since $\wilde\alpha$ is Morita equivalent to the dual coaction $\what\alpha$,
on the relative commutant.
We write
\begin{align*}
\fix(A,\alpha,\mu)&=C(A\rtimes_\alpha G,\mu\rtimes G)
\\
\delta^\mu&=C(\wilde\alpha),
\end{align*}
so that the functor $\fix$ takes an object $(A,\alpha,\mu)$ to $(\fix(A,\alpha,\mu),\delta^\mu)$.
$\fix$ is given on morphisms by
\[
\fix(\phi)=C(\phi\rtimes G).
\]
It is important to keep in mind that the $C^*$-algebra $\fix(A,\alpha,\mu)$ lies in the multiplier algebra of the crossed product $A\rtimes_\alpha G$,
and that $\delta^\mu$ is a maximal coaction on this algebra
(because $C(\wilde\alpha)$ is maximal, as we mentioned above),
We will soon motivate the choice of the word ``Fix''.

\begin{thm}[{\cite{cldx}}]\label{cpc thm}
Let $\cpc\:\co\to\eqac$ be the functor
with object map
$(A,\delta)\to (A\rtimes_\delta G,\what\delta,j_G)$ defined in \secref{categories}, and let $\cpcm\:\co^m\to\eqac$ be the restriction to the subcategory of maximal coactions.
Then $\cpcm$ is a category equivalence, with quasi-inverse $\fix$.
\end{thm}

\begin{proof}

We will show that both compositions $\fix\,\circ\cpcm$ and $\cpcm\circ\,\fix$ are naturally isomorphic to the identity functors $\id_{\co^m}$ and $\id_{\eqac}$, respectively.
Since $\fix\circ\cpc$ is Fischer's construction of the maximalization functor $\ma$, and since the surjections $\psi$ give a natural transformation from $\ma$ to $\id_{\co^m}$, and finally since $\psi_A$ is an isomorphism whenever $(A,\delta)$ is maximal, we see that $\fix\circ\cpcm\simeq\id_{\co^m}$.

The other natural isomorphism can be cobbled together from results in the literature;
here we give one such cobbling.
We introduce some auxiliary functors:
\begin{itemize}
\item $\cpc^n\:\co^n\to \eqac$ is the restriction of $\cpc$ to the normal coactions;
\item $\fix^n\:\eqac\to \co^n$ is the quasi-inverse of $\cpc^n$ defined in \cite[Section~6.1]{koqlandstad};
\item $\ma^n\:\co^n\to \co^m$ is the restriction of $\ma$ to the normal coactions;
\item $\nor^m\:\co^m\to \co^n$ is the restriction of $\nor$ to the maximal coactions.
\end{itemize}
Consider the following diagram of functors:
\[
\xymatrix@C+30pt@R+20pt{
\co^m \ar@/^/[r]^-{\cpc^m} \ar@/_/[d]_{\nor^m}
&\eqac \ar@{=}[d] \ar@/^/[l]^-{\fix}
\\
\co^n \ar@/_/[u]_{\ma^n} \ar@/^/[r]^-{\cpc^n}
&\eqac. \ar@/^/[l]^-{\fix^n}
}
\]
We know from \cite[Theorem~3.3]{clda} that $\nor^m$ and $\ma^n$ are quasi-inverses,
and from \cite[Theorems~4.1 and 5.1 and their proofs]{clda} that $\cpc^n$ and $\fix^n$ are quasi-inverses.
We have shown above that $\fix\circ \cpc^m\simeq \id_{\co^m}$,
and we have $\ma^n=\fix\circ \cpc^n$ by Definitions~\ref{max def} and \ref{fix functor}.
Thus $\fix\simeq \ma^n\circ\fix^n$ since $\cpc^n\circ\fix^n\simeq \id_{\eqac}$.
The desired natural isomorphism $\cpc^m\circ\,\fix\simeq \id_{\eqac}$ now follows by routine computations with the functors:
since $\fix\circ \cpc^n=\ma^n$,
we have
\begin{align*}
\cpc^n\circ\nor^m\circ\,\fix
&\simeq \cpc^n\circ\nor^m\circ\ma^n\circ\fix^n
\\&\simeq \cpc^n\circ\,\fix^n
\\&\simeq \id_{\eqac},
\end{align*}
and so
\begin{align*}
\cpc^m\circ\,\fix
&\simeq \cpc^n\circ\nor^m\circ\,\fix\circ\cpc^m\circ\,\fix
\\&\simeq \cpc^n\circ\nor^m\circ\,\fix
\\&\simeq \id_{\eqac}.
\qedhere
\end{align*}
\end{proof}

\begin{defn}
Motivated by 
\thmref{cpc thm},
we call $\fix(A,\alpha,\mu)$ the
\emph{maximal generalized fixed-point algebra} of the equivariant action $(A,\alpha,\mu)$.
\end{defn}

In \cite{koqlandstad},
we used slightly different notation for some of the auxiliary functors of the above proof,
namely $\wcpn$ for $\cpc^n$, and $\fixn$ for $\fix^n$.
Note that although the functor $\fix$ of \defnref{fix functor}
gives generalized fixed-point algebras
for equivariant nondegenerate categories,
it is not the same as the functor $\fixn$ defined in \cite[Subsection~6.1]{koqlandstad} --- the latter produces normal coactions, while $\fix$ produces maximal ones.

\begin{rem}
Part of \thmref{cpc thm},
namely that $\cpc^m\:\co^m\to \eqac$ is an equivalence,
has been recorded in the literature several times.
The reason we had to do some work was that the specific quasi-inverse $\fix$ has not been so well studied (in fact we could not find a reference where it is used in this particular way).
Thus, for example, it would not be enough to simply refer to \cite[Corollary~4.3]{cldx}, since the particular quasi-inverse $\fix$ we use here does not appear there.
\end{rem}

\begin{rem}
In the proof of \thmref{cpc thm} we observed that $\fix\simeq\ma^n\circ\fix^n$.
Armed with the result of \thmref{cpc thm} itself we also see that $\fix^n\simeq \nor^m\circ\fix$.
It follows that the \emph{normal generalized fixed-point algebra} $\fix^n(A,\alpha,\mu)$ of an equivariant action,
which in \cite{koqlandstad} we regarded as a subalgebra of $M(A)$,
can alternatively be regarded as a nondegenerate subalgebra of $M(A\rtimes_{\alpha,r} G)$,
similarly to how we regard the maximal generalized fixed-point algebra $\fix(A,\alpha,\mu)$ as a subalgebra of $M(A\rtimes_\alpha G)$.
This is consistent with the well-known fact that the dual coaction on the reduced crossed product $A\rtimes_{\alpha,r} G$ is a normalization of the (maximal) dual coaction on the full crossed product $A\rtimes_\alpha G$.
\end{rem}

\begin{rem}
Reasoning quite similar to that in the proof of \cite[Proposition~5.1]{koqlandstad} shows that the category equivalence recorded in \thmref{cpc thm} gives a \emph{good inversion},
in the sense of \cite[Definition~4.1]{koqlandstad},
of the crossed-product functor
$\co^m\to\cs$
defined on objects by
\[
(A,\delta)\mapsto A\rtimes_\delta G.
\]
\end{rem}

\begin{cor}\label{cpa equiv}
The functor $\cpa\:\eqac\to \kco^m$ is a category equivalence.
\end{cor}

\begin{proof}
By \thmref{cpc thm} we have an equivalence $\fix=\dst\circ \cpa$.
By \cite[Theorem~4.4]{koqstable}, $\dst$ is an equivalence when regarded as a functor between the categories $\kalg$ and $\cs$,
and the stabilization functor $\cs\to \kalg$,
given on objects by $A\mapsto A\otimes\KK$,
is a quasi-inverse.
Since a morphism in either category is an isomorphism if and only if it is an isomorphism between the $C^*$-algebras,
it follows that the destabilization functor $\dst\:\kco\to \co$ is also an equivalence,
with quasi-inverse given on objects by $(A,\delta)\mapsto (A\otimes\KK,\delta\otimes_*\id)$.
By \lemref{coact commutant}, $\delta$ is maximal if and only if $C(\delta)$ is.
Thus $\dst$ restricts to an equivalence between $\kco^m$ and $\co^m$.
Therefore $\cpa\:\eqac\to \kco^m$ is also an equivalence.
\end{proof}

\section{Enchilada categories and functors}\label{enchilada categories}

\subsection*{Categories}

We again recall some categories from \cite{koqlandstad} and \cite{koqstable} (and \cite{enchilada}), and introduce a few more.
Everything will be based upon
the \emph{enchilada category $\cse$ of $C^*$-algebras}, where a morphism $[X]\:A\to B$ is the isomorphism class of a nondegenerate $A-B$ correspondence $X$.

The \emph{enchilada category $\coe$ of coactions}
has the same objects as $\co$,
and a morphism $[X,\zeta]\:(A,\delta)\to (B,\epsilon)$ in $\coe$
is the isomorphism class of a nondegenerate $A-B$ correspondence $X$
equipped with a $\delta-\epsilon$ compatible coaction $\zeta$.
Composition of morphisms is given by
\[
[Y,\eta]\circ [X,\zeta]=[X\otimes_B Y,\zeta\cotimes_B\eta].
\]
$\coe^m$ denotes the full subcategory of $\coe$ whose objects are the maximal coactions.

The \emph{enchilada category $\ace$ of actions}
has the same objects as $\ac$,
and a morphism $[X,\gamma]\:(A,\alpha)\to (B,\beta)$ in $\ace$
is the isomorphism class of a nondegenerate $A-B$ correspondence $X$
equipped with an $\alpha-\beta$ compatible action $\gamma$.

The \emph{enchilada category $\eqace$ of equivariant actions}
has the same objects as $\eqac$,
and a morphism $[X,\gamma]\:(A,\alpha,\mu)\to (B,\beta,\nu)$ in $\eqace$
is a morphism $[X,\gamma]\:(A,\alpha)\to (B,\beta)$ in $\ace$
(with no assumptions relating $(X,\gamma)$ to $\mu$ or $\nu$).

$\cse$, $\ace$, and $\coe$
had their origins in \cite[Theorems~2.2, 2.8, and 2.15, respectively]{enchilada}.
In \cite[Subsection~6.2]{koqlandstad} we used the notation $\wace$ for $\eqace$.

The \emph{enchilada category $\kalge$ of $\KK$-algebras}
has the same objects as $\kalg$,
and a morphism $[X]\:(A,\iota)\to (B,\jmath)$ in $\kalge$
is a morphism $[X]\:A\to B$ in $\cse$.
In \cite[Definition~6.1]{koqstable} we used the notation $\kalgen$ for $\kalge$.

The \emph{enchilada category $\kcoe$ of $\KK$-coactions}
has the same objects as $\kco$,
and when we say $[X,\zeta]\:(A,\delta,\iota)\to (B,\epsilon,\jmath)$ is a morphism in $\kcoe$
we mean that $[X,\zeta]\:(A,\delta)\to (B,\epsilon)$ is a morphism in $\coe$.
Note that this implies in particular that $[X]\:A\to B$ is a morphism in $\cse$, and hence $[X]\:(A,\iota)\to (B,\jmath)$ is a morphism in $\kalge$.
It is routine to check that this is a category.

Note that the above notations $\eqace$, $\kalge$, and $\kcoe$ are potentially misleading: they are not coslice categories. The notation was chosen to indicate some parallel with the nondegenerate versions $\eqac$, $\kalg$, and $\kco$.
In fact, $\eqace$ is a \emph{semi-comma category} in the sense of
\cite{hkrwnatural}.

\subsection*{Functors}

We define a functor
\[
\cpce\:\coe\to\eqace
\]
to be the same as $\cpc$ on objects,
and on morphisms as follows:
if $[X,\zeta]\:(A,\delta)\to (B,\epsilon)$ is a morphism in $\coe$ then
\[
\cpce[X,\zeta]\:(A\rtimes_\delta G,\what\delta,j_G^\delta)\to
(B\rtimes_\epsilon G,\what\epsilon,j_G^\epsilon)
\]
is the morphism in $\eqace$ given by $\cpce[X,\zeta]=[X\rtimes_\zeta G,\what\zeta]$.

In order to define a suitable functor $\eqace\to \kcoe$,
we first need to see how to manipulate the coactions:

\begin{lem}\label{gamma tilda}
Let $[X,\gamma]\:(A,\alpha,\mu)\to (B,\beta,\nu)$ be a morphism in $\eqace$.
Then there is an $\wilde\alpha-\wilde\beta$ compatible coaction $\wilde\gamma$ on the $(A\rtimes_\alpha G)-(B\rtimes_\beta G)$ correspondence $X\rtimes_\gamma G$ given by
\[
\wilde\gamma(y)=V_A\what\gamma(y)V_B^*\righttext{for}y\in X\rtimes_\gamma G,
\]
where $V_A=((i_A\circ\mu)\otimes\id)(w_G)$ is the $\what\alpha$-cocycle from \lemref{cocycle}, and similarly for $V_B$.
\end{lem}

\begin{proof}
Let $K=\KK(X)$ be the algebra of compact operators on the Hilbert $B$-module $X$,
with left-$A$-module homomorphism $\varphi_A\:A\to M(K)$,
and let $L=L(X)=\smtx{K&X\\{*}&B}$ be the linking algebra.
Let $\sigma$ and $\tau=\smtx{\sigma&\gamma\\{*}&\beta}$ be the associated actions of $G$ on $K$ and $L$, respectively.
Then $\kappa:=\varphi_A\circ\mu\:C_0(G)\to M(K)$
and $\omega:=\smtx{\kappa&0\\0&\nu}\:C_0(G)\to M(L)$ are equivariant.
The crossed product decomposes as
\[
L\rtimes_\tau G=\mtx{K\rtimes_\sigma G&X\rtimes_\gamma G\\{*}&B\rtimes_\beta G},
\]
the dual coaction as
\[
\what\tau=\mtx{\what\sigma&\what\gamma\\{*}&\what\beta},
\]
and the $\what\tau$-cocycle as
\[
V_L=\mtx{V_K&0\\0&V_B}.
\]
Thus the perturbed dual coaction decomposes as
\begin{align*}
\wilde\tau
&=\ad V_L\circ\what\tau
\\&=\mtx{V_K&0\\0&V_B}\mtx{\what\sigma&\what\gamma\\{*}&\what\beta}\mtx{V_K^*&0\\0&V_B^*}
\\&=\mtx{\ad V_K\circ\what\sigma&V_K\what\gamma V_B^*\\{*}&\ad V_B\circ\what\beta},
\end{align*}
and since it preserves the corner projections it compresses on the upper-right corner to a coaction on the Hilbert $(K\rtimes_\sigma G)-(B\rtimes_\beta G)$ bimodule $X\rtimes_\gamma G$, given by
\[
y\mapsto V_K\what\gamma(y)V_B^*\righttext{for}y\in X\rtimes_\gamma G.
\]
Using the canonical isomorphism $\KK(X)\rtimes_\sigma G\simeq \KK(X\rtimes_\gamma G)$,
one readily checks that
for all $y\in X\rtimes_\gamma G$ we have
\[
V_K\what\gamma(y)
=\bigl((\varphi_A\rtimes G)\otimes\id\bigr)(V_A)\what\gamma(y)
=V_A\what\gamma(y).
\]

Finally, we check that $\wilde\gamma$ is compatible with the left $(A\rtimes_\alpha G)$-module structure:
for $d\in A\rtimes_\alpha G$ and $y\in X\rtimes_\gamma G$,
since the Hilbert-module homomorphism
\[
\wilde\gamma\:X\rtimes_\gamma G\to M\bigl((X\rtimes_\gamma G)\otimes C^*(G)\bigr)
\]
is nondegenerate
and the homomorphism
\[
\varphi_A\rtimes G\:A\rtimes_\alpha G\to M(K\rtimes_\sigma G)
\]
is $\wilde\delta-\wilde\sigma$ equivariant,
we have
\begin{align*}
\wilde\gamma(dy)
&=\wilde\gamma\bigl((\varphi_A\rtimes G)(d)y\bigr)
\\&=\wilde\sigma\circ(\varphi_A\rtimes G)(d)\wilde\gamma(y)
\\&=\bigl((\varphi_A\rtimes G)\otimes\id\bigr)\circ\wilde\delta(d)\wilde\gamma(y)
\\&=\wilde\delta(d)\wilde\gamma(y)
\qedhere
\end{align*}
\end{proof}

\begin{cor}\label{cpaen}
There is a functor
\[
\cpae\:\eqace\to \kcoe
\]
with the same object map as $\cpa$,
and
given
on morphisms
as follows: if
$[X,\gamma]\:(A,\alpha,\mu)\to (B,\beta,\nu)$
is a morphism in $\eqace$
then
\[
\cpae[X,\gamma]\:
(A\rtimes_\alpha G,\wilde\alpha,\mu\rtimes G)\to (B\rtimes_\beta G,\wilde\beta,\nu\rtimes G)
\]
is the morphism in $\kcoe$ given by 
$\cpae[X,\gamma]=[X\rtimes_\gamma G,\wilde\gamma]$,
where $\wilde\gamma$ is the
$\wilde\alpha-\wilde\beta$ compatible coaction on $X\rtimes_\gamma G$ defined in \lemref{gamma tilda}.
\end{cor}

\begin{proof}
We start with the crossed-product functor from $\ace$ to $\coe$
that is
given on objects by
\[
(A,\alpha)\mapsto (A\rtimes_\alpha G,\what\alpha),
\]
and that takes a morphism
$[X,\gamma]\:(A,\alpha)\to (B,\beta)$
in $\ace$ to
the morphism
\[
[X\rtimes_\gamma G,\what\gamma]\:(A\rtimes_\alpha G,\what\alpha)\to (B\rtimes_\beta G,\what\beta)
\]
in $\coe$.
Now, to say $[X,\gamma]\:(A,\alpha,\mu)\to (B,\beta,\nu)$ is a morphism in $\eqace$ just means that
$[X,\gamma]\:(A,\alpha)\to (B,\beta)$ is a morphism in $\ace$,
and then
$[X\rtimes_\gamma G,\what\gamma]\:(A\rtimes_\alpha G,\what\alpha,\mu\rtimes G)\to
(B\rtimes_\beta G,\what\beta,\nu\rtimes G)$
is a morphism in $\kcoe$
because
$[X\rtimes_\gamma G,\what\gamma]\:(A\rtimes_\alpha G,\what\alpha)\to (B\rtimes_\beta G,\what\beta)$
is a morphism in $\coe$.
Since $[X,\gamma]\mapsto [X\rtimes_\gamma G,\what\gamma]$
is functorial from $\ace$ to $\coe$,
it is also functorial from $\eqace$ to $\kcoe$,
because composition of morphisms in $\eqace$ and $\kcoe$
is the same as in $\ace$ and $\coe$, respectively.

Now, we actually need to use the perturbed coaction $\wilde\alpha$ instead of the dual coaction $\what\alpha$.
From \lemref{gamma tilda} it follows that if
$[X,\gamma]\:(A,\alpha,\mu)\to (B,\beta,\nu)$
is a morphism in $\eqace$
then
$[X\rtimes_\gamma G,\wilde\gamma]\:(A\rtimes_\alpha G,\wilde\alpha,\mu\rtimes G)\to
(B\rtimes_\beta G,\wilde\beta,\nu\rtimes G)$
is a morphism in $\kcoe$.
The assignments
\[
[X,\gamma]\mapsto [X\rtimes_\gamma G,\wilde\gamma]
\]
are functorial for the correspondences $X\rtimes_\gamma G$,
and we must show that they are functorial for the coactions $\wilde\gamma$.
Given another morphism
$[Y,\tau]\:(B,\beta,\nu)\to (C,\sigma,\omega)$
in $\eqace$,
by \cite[Proposition~3.8 and its proof]{taco}
there is an isomorphism
\[
\Upsilon\:(X\rtimes_\gamma G)\otimes (Y\rtimes_\tau G)\variso
(X\otimes_B Y)\rtimes_{\gamma\otimes\tau} G
\]
of $(A\rtimes_\alpha G)-(C\rtimes_\sigma G)$ correspondences,
that takes an elementary tensor $x\otimes y$
for $x\in C_c(G,X)$ and $y\in C_c(G,Y)$
to the function $\Upsilon(x\otimes y)\in C_c(G,X\otimes_B Y)$ given by
\[
\Upsilon(x\otimes y)(s)=\int_G x(t)\otimes \tau_t\bigl(y(t\inv s)\bigr)\,dt.
\]
Computations similar to those in the proof of \cite[Theorem~3.7]{enchilada}
(which used reduced crossed products) show that $\Upsilon$ is
$(\what\gamma\cotimes_{B\rtimes_\beta G} \what\tau)-\what{\gamma\otimes_B \tau}$ equivariant.
For the
$(\wilde\gamma\cotimes_{B\rtimes_\beta G} \wilde\tau)-\wilde{\gamma\otimes_B \tau}$ equivariance,
we compute,
for $x\in X\rtimes_\gamma G$ and $y\in Y\rtimes_\tau G$,
\begin{align*}
\wilde{\gamma\otimes_B \tau}\circ \Upsilon(x\otimes y)
&=V_A\what{\gamma\otimes_B \tau}\bigl(\Upsilon(x\otimes y)\bigr)V_C^*
\\&=V_A(\Upsilon\otimes\id)\circ (\what\gamma\cotimes_{B\rtimes_\beta G} \what\tau)(x\otimes y)V_C^*
\\&\overset{(1)}{=}(\Upsilon\otimes\id)\bigl(V_A(\what\gamma\cotimes_{B\rtimes_\beta G} \what\tau)(x\otimes y)V_C^*\bigr)
\\&=(\Upsilon\otimes\id)\Bigl(V_A\Theta\bigl(\what\gamma(x)\otimes \what\tau(y)\bigr)V_C^*\Bigr)
\\&\overset{(2)}{=}(\Upsilon\otimes\id)\Bigl(\Theta\bigl(V_A(\what\gamma(x)\otimes \what\tau(y))V_C^*\bigr)\Bigr)
\\&=(\Upsilon\otimes\id)\Bigl(\Theta\bigl(V_A\what\gamma(x)\otimes \what\tau(y)V_C^*\bigr)\Bigr)
\\&\overset{(3)}{=}(\Upsilon\otimes\id)\Bigl(\Theta\bigl(V_A\what\gamma(x)V_B^*\otimes V_B\what\tau(y)V_C^*\bigr)\Bigr)
\\&=(\Upsilon\otimes\id)\Bigl(\Theta\bigl(\wilde\gamma(x)\otimes \wilde\tau(y)\bigr)\Bigr)
\\&=(\Upsilon\otimes\id)\circ (\wilde\gamma\cotimes_{B\rtimes_\beta G} \what\tau)(x\otimes y),
\end{align*}
where the equalities at (1) and (2) follow since $\Upsilon\otimes\id$ and $\Theta$ are
homomorphisms of $((A\rtimes_\alpha G)\otimes C^*(G))-((C\rtimes_\sigma G)\otimes C^*(G))$ correspondences,
and
the equality at (3) since
$((X\rtimes_\gamma G)\otimes C^*(G))\otimes_{(B\rtimes_\beta G)\otimes C^*(G)} ((Y\rtimes_\tau G)\otimes C^*(G))$
is balanced over $(B\rtimes_\beta G)\otimes C^*(G)$.
It now follows similarly to the preceding that the assignment
$[X,\gamma]\mapsto [X\rtimes_\gamma G,\wilde\gamma]$
is functorial.
\end{proof}

The proof of \cite[Theorem~6.4]{koqstable}
(see also
\cite[Proposition~5.1]{koqstable}) shows that if $[X]\:(A,\iota)\to (B,\jmath)$ is a morphism in $\kalge$ then
\[
C(X,\iota,\jmath):=\{x\in M(X):\iota(k)x=x\jmath(k)\in X\text{ for all }k\in \KK\}
\]
is a nondegenerate $C(A,\iota)-C(B,\jmath)$ correspondence,
and there is a (necessarily unique) isomorphism
\[
\theta_X\:C(X,\iota,\jmath)\otimes\KK\variso X
\]
as $A-B$ correspondences
(where we regard
$C(X,\iota,\jmath)$ as an $A-B$ correspondence
via the isomorphisms
$\theta_A\:C(A,\iota)\otimes\KK\variso A$
and
$\theta_B\:C(B,\jmath)\otimes\KK\variso B$)
such that
\[
\theta_X(x\otimes k)=x\jmath(k)\righttext{for all}x\in C(X,\iota,\jmath),k\in\KK.
\]
It follows from \cite[Theorem~6.4 and its proof]{koqstable} that there is a functor
$\kalge\to\cse$
that is defined on objects by $(A,\iota)\mapsto C(A,\iota)$
and
that 
takes a morphism
$[X]\:(A,\iota)\to (B,\jmath)$ in $\kalge$
to the morphism
\[
[C(X,\iota,\jmath)]\:C(A,\iota)\to (B,\jmath)
\]
in $\cs$,
and 
which
moreover 
is a category equivalence,
with quasi-inverse given by
the 
\emph{enchilada stabilization functor}
taking $A$ to $(A\otimes\KK,1\otimes\id)$ and a
morphism $[X]\:A\to B$ in $\cse$
to the morphism
$[X\otimes\KK]$,
where $X\otimes\KK$ is the external-tensor-product $(A\otimes\KK)-(B\otimes\KK)$ correspondence.
In \cite[Definition~6.3]{koqstable} 
the enchilada stabilization functor
was denoted by $\wste$,
and 
its
quasi-inverse
was not 
given a name (although the nondegenerate version was denoted by $\dstn$ in
\cite{koqstable}, as explained at the end of \secref{categories}).
We want an equivariant version,
and we need to see what to do with the coactions:

\begin{lem}\label{com zeta}
Let $[X,\zeta]\:(A,\delta,\iota)\to (B,\epsilon,\jmath)$ be a morphism in $\kcoe$.
Then there is a $C(\delta)-C(\epsilon)$ compatible coaction $C(\zeta)$ on the $C(A,\iota)-C(B,\jmath)$ correspondence $C(X,\iota,\jmath)$ given by
the restriction to $C(X,\iota,\jmath)$ of the canonical extension of $\zeta$ to $M(X)$.
\end{lem}

\begin{proof}
Again let $K=\KK(X)$ be the algebra of compact operators on the Hilbert $B$-module $X$,
with left-$A$-module homomorphism $\varphi_A\:A\to M(K)$,
and let $L=L(X)=\smtx{K&X\\{*}&B}$ be the linking algebra.
Let $\eta$ and $\xi=\smtx{\eta&\zeta\\{*}&\epsilon}$ be the associated coactions of $G$ on $K$ and $L$, respectively.
Then $\kappa:=\varphi_A\circ\iota\:\KK\to M(K)$
and $\omega:=\smtx{\kappa&0\\0&\jmath}\:\KK\to M(L)$ are equivariant for the trivial coaction on $\KK$.
Note that $[X,\zeta]\:(K,\eta,\kappa)\to (B,\epsilon,\jmath)$ is a morphism in $\kcoe$,
and
\[
C(X,\kappa,\jmath)=C(X,\iota,\jmath).
\]
By \cite[Proposition~5.1 and Theorem~6.4 and their proofs]{koqstable},
the relative commutant decomposes as
\[
C(L,\omega)=\mtx{C(K,\kappa)&C(X,\iota,\jmath)\\{*}&C(B,\jmath)}.
\]
A routine computation shows that the associated coaction
preserves the corner projections,
and hence
compresses on the upper-right corner to a coaction $C(\zeta)$ on the Hilbert $C(K,\kappa)-C(B,\jmath)$ bimodule $C(X,\kappa,\jmath)$, given by
the restriction to $C(X,\iota,\jmath)$ of the extension of $\zeta$ to $M(X)$.

It remains to check that $C(\zeta)$ is compatible with the left $C(A,\iota)$-module structure:
for $a\in C(A,\iota)$ and $x\in C(X,\iota,\jmath)=C(X,\kappa,\jmath)$,
since the Hilbert-module homomorphism
\[
C(\zeta)\:C(X,\kappa,\jmath)\to M\bigl(C(X,\kappa,\jmath)\otimes C^*(G)\bigr)
\]
is nondegenerate
and the homomorphism
\[
\varphi_A\:A\to M(K)
\]
is $\delta-\eta$ equivariant,
we have
\begin{align*}
C(\zeta)(ax)
&=C(\zeta)\bigl(\varphi_A(a)x\bigr)
\\&=C(\eta)\circ\varphi_A(a)C(\zeta)(x)
\\&=\eta\circ\varphi_A(a)C(\zeta)(x)
\\&=(\varphi_A\otimes\id)\circ\delta(a)C(\zeta)(x)
\\&=\delta(a)C(\zeta)(x)
\\&=C(\delta)(a)C(\zeta)(x).
\qedhere
\end{align*}
\end{proof}

Thus we can define a functor
\[
\dste\:\kcoe\to\coe
\]
to be the same as $\dst$ on objects,
and
on morphisms
as follows: if
$[X,\zeta]\:(A,\delta,\iota)\to (B,\epsilon,\jmath)$
is a morphism in $\kcoe$
then
\[
\dste[X,\zeta]\:
\bigl(C(A,\iota),C(\delta)\bigr)\to
\bigl(C(B,\jmath),C(\epsilon)\bigr)
\]
is the morphism in $\coe$ given by 
\[
\dste[X,\zeta]=[C(X,\iota,\jmath),C(\zeta)],
\]
where $C(\zeta)$ is the
$C(A,\iota)-C(B,\jmath)$ compatible coaction on $C(X,\iota,\jmath)$ defined in \lemref{com zeta}.

\begin{cor}\label{comcoen}
The 
above functor $\dste\:\kcoe\to\coe$
is a category equivalence,
and moreover it
restricts to an equivalence between $\kcoe^m$ and $\coe^m$.
\end{cor}

\begin{proof}
For the first part, we need to know that $\dste$
is essentially surjective, full, and faithful.
Essential surjectivity is trivial, since $\dst\:\kco\to\co$ is an equivalence,
and since isomorphism in $\co$ is stronger than in the category $\coe$.

Since 
the
(non-equivariant)
enchilada destabilization functor $\kalge\to\cse$
is an equivalence,
and since the morphisms in $\kcoe$ and $\coe$ are just morphisms in $\kalge$ and $\cse$, respectively, that preserve the extra structure,
$\dste$ is full and faithful.

For the other part, we only need to recall from \lemref{coact commutant} that if $(A,\delta,\iota)$ is an object in $\kcoe$
then the coaction $\delta$ is maximal if and only if $C(\delta)$ is.
\end{proof}

\section{Enchilada Landstad duality}\label{enchilada duality}

In this section we prove categorical Landstad duality for
maximal coactions and
the enchilada categories.

We define a functor $\fixe$ by the commutative diagram
\[
\xymatrix@C+20pt{
\eqace \ar[r]^{\cpae} \ar[dr]_{\fixe}
&\kcoe \ar[d]^{\dste}
\\
&\coe^m.
}
\]
Thus $\fixe$ is the same as $\fix$ on objects,
and if $[X,\gamma]\:(A,\alpha,\mu)\to (B,\beta,\nu)$ is a morphism in $\eqace$
then
\[
\fixe[X,\gamma]\:(\fix(A,\alpha,\mu),\delta^\mu)\to (\fix(B,\beta,\nu),\delta^\nu)
\]
is the morphism in $\coe$ given by
\[
\fixe[X,\gamma]=[C(X\rtimes_\gamma G,j_G^\delta,j_G^\epsilon),C(\wilde\gamma)].
\]
We write
\begin{align*}
\fix(X,\gamma,\mu,\nu)&=C(X\rtimes_\gamma G,j_G^\delta,j_G^\epsilon)
\\
\delta^{\mu,\nu}&=C(\wilde\gamma).
\end{align*}

\begin{thm}\label{cpcen thm}
Let $\cpce\:\coe\to\eqace$ be the functor from \secref{enchilada categories},
and let $\cpcem\:\coe^m\to \eqace$ be the restriction to the subcategory of maximal coactions.
Then $\cpcem$ is a category equivalence, with quasi-inverse $\fixe$.
\end{thm}

\begin{proof}
The proof is similar to that of \cite[Theorem~6.2]{koqlandstad}, which in turn was based upon the proof of \cite[Theorem~5.2]{koqlandstad}.
The functor $\cpcem$ is essentially surjective because $\cpcm$ is,
and isomorphism in $\eqac$ is stronger than isomorphism in $\eqace$.

To see that $\cpcem$ is full,
it suffices to show that if $[X,\gamma]\:(A,\alpha,\mu)\to (B,\beta,\nu)$ is a morphism in $\eqace$,
then
\[
\bigl(\fix(X,\gamma,\mu,\nu)\rtimes_{\delta^{\mu,\nu}} G,\what{\delta^{\mu,\nu}}\bigr)\simeq
(X,\gamma)
\]
as $(A,\alpha)-(B,\beta)$ correspondences,
where we regard
$\fix(X,\gamma,\mu,\nu)\rtimes_{\delta^{\mu,\nu}} G$ as an $A-B$ correspondence
via the isomorphisms
\begin{align*}
\Theta_A&\:\fix(A,\alpha,\mu)\rtimes_{\delta^\mu} G\variso A
\\
\Theta_B&\:\fix(B,\beta,\nu)\rtimes_{\delta^\nu} G\variso B.
\end{align*}
The verification that we give below consists of routine linking-algebra computations, along the lines of \cite[Propositions~6.2 and 5.2]{koqlandstad}.
The primary difference here is that we already have the functor $\fixe$ in hand.

Recall the notation $K,L,\varphi_A$ from the discussion preceding \lemref{com zeta}.
Let $\sigma$ be the associated action on $K$
and $\kappa=\varphi_A\circ\mu\:C_0(G)\to M(K)$,
and let $\tau=\smtx{\sigma&\gamma\\{*}&\beta}$ and $\omega=\smtx{\kappa&0\\0&\nu}\:C_0(G)\to M(L)$,
giving equivariant actions
$(K,\sigma,\kappa)$ and $(L,\tau,\omega)$.

We want to show that
\begin{align}\label{fix link}
\begin{split}
\fix(L,\tau,\omega)
&=\mtx{\fix(K,\sigma,\kappa)&\fix(X,\gamma,\mu,\nu)\\
{*}&\fix(B,\beta,\nu)}
\\
\delta^\omega&=\mtx{\delta^\kappa&\delta^{\mu,\nu}\\{*}&\delta^\nu}.
\end{split}
\end{align}
We will use the decomposition $\fixe=\dste\circ\cpae$.
First
observe that
\begin{align*}
L\rtimes_\tau G
&=\mtx{K\rtimes_\sigma G&X\rtimes_\gamma G\\
{*}&B\rtimes_\beta G}
\\
\what\tau&=\mtx{\what\sigma&\what\gamma\\{*}&\what\beta}
\\
i_L&=\mtx{i_K&i_X\\{*}&i_B}
\\
i_G^\tau&=\mtx{i_G^\sigma&0\\0&i_G^\beta}.
\end{align*}
This follows from
\cite[proof of Proposition~3.5]{taco}
and
\cite[Proposition~3.5]{enchilada}.
\cite{taco} deals with full crossed but does not handle the dual coactions, while \cite{enchilada} handles the dual coactions but only for reduced crossed products; the techniques of \cite{enchilada} carry over to full crossed products with no problem.
It follows that
\begin{align*}
V_L
&=\bigl((i_L\circ\omega)\otimes\id\bigr)(w_G)
=\mtx{V_K&0\\0&V_B},
\end{align*}
so that the perturbations of the dual coactions satisfy
\[
\wilde\tau=\mtx{\wilde\sigma&\wilde\gamma\\{*}&\wilde\beta}.
\]
It further follows that
\[
\omega\rtimes G=\mtx{\kappa\rtimes G&0\\0&\beta\rtimes G}.
\]

Combining the above with
\lemref{com zeta} and the discussion preceding it, and
with
\corref{comcoen},
the equalities \eqref{fix link} can be
justified with the following calculations.
\begin{align*}
&\fix(L,\tau,\omega)
\\&\quad=C(L\rtimes_\tau G,\wilde\tau,\omega\rtimes G)
\\&\quad=C(L(X\rtimes_\gamma G),\wilde\tau,\omega\rtimes G)
\\&\quad=\mtx{
C(\KK(X\rtimes_\gamma G),\wilde\sigma,\kappa\rtimes G)
&C(X\rtimes_\gamma G,\wilde\gamma,\kappa\rtimes G,\nu\rtimes G)
\\
{*}&C(B\rtimes_\beta G,\wilde\beta,\nu\rtimes G)
}
\\&\quad=\mtx{
C(K\rtimes_\sigma G,\wilde\sigma,\kappa\rtimes G)
&C(X\rtimes_\gamma G,\wilde\gamma,\kappa\rtimes G,\nu\rtimes G)
\\
{*}&C(B\rtimes_\beta G,\wilde\beta,\nu\rtimes G)
}
\\&\quad=\mtx{\fix(K,\sigma,\kappa)&\fix(X,\gamma,\kappa,\nu)
\\{*}&\fix(B,\beta,\nu)}
\end{align*}
and
\begin{align*}
\delta^\omega
&=C(\wilde\tau)
=\mtx{C(\wilde\sigma)&C(\wilde\gamma)
\\{*}&C(\wilde\beta)}
=\mtx{\delta^\kappa&\delta^{\mu,\nu}\\{*}&\delta^\nu}.
\end{align*}

We have an isomorphism
\[
\Theta_L\:
\bigl(\fix(L,\tau,\omega)\rtimes_{\delta^\omega} G,\what{\delta^\omega},j_G^{\delta^\omega}\bigr)
\variso (L,\tau,\omega)
\]
On the other hand,
\cite[Proposition~3.9 and Theorem~3.13]{enchilada} give
\begin{align*}
\fix(L,\tau,\omega)\rtimes_{\delta^\omega} G
&=\mtx{
\fix(L,\tau,\omega)\rtimes_{\delta^\omega} G
&\fix(K,\sigma,\kappa)\rtimes_{\delta^\kappa} G
\\{*}&\fix(B,\beta,\nu)\rtimes_{\delta^\nu} G}
\\
\what{\delta^\omega}
&=\mtx{\what{\delta^\kappa}&\what{\delta^{\mu,\nu}}
\\{*}&\what{\delta^\nu}}.
\end{align*}
Since $\Theta_L$ preserves the corner projections, it restricts on the corners to a
$\what{\delta^{\mu,\nu}}-\gamma$ equivariant
Hilbert-bimodule isomorphism
\begin{align*}
&(\Theta_K,\Theta_X,\Theta_B)\:
\\&\quad\bigl(\fix(K,\sigma,\kappa)\rtimes_{\delta^\kappa} G,
\fix(X,\gamma,\kappa,\nu)\rtimes_{\delta^{\mu,\nu}} G,
\fix(B,\beta,\nu)\rtimes_{\delta^\nu} G\bigr)
\\&\hspace{1in}
\variso
(K,X,B).
\end{align*}

We also have a
$\what{\delta^\mu}-\alpha$ equivariant isomorphism
\[
\Theta_A\:\fix(A,\alpha,\mu)\rtimes_{\delta^\mu} G\variso A,
\]
and the diagram
\[
\xymatrix@C+20pt{
\fix(A,\alpha,\mu)\rtimes_{\delta^\mu} G
\ar[r]^-{\Theta_A}_-\simeq \ar[d]_{\fix(\varphi_A)\rtimes G}
&A \ar[d]^{\varphi_A}
\\
\fix(K,\sigma,\kappa)\rtimes_{\delta^\kappa} G
\ar[r]_-{\Theta_K}^-\simeq
&K
}
\]
of morphisms in $\cs$
commutes by nondegenerate Landstad duality.
Thus, incorporating the isomorphisms $\Theta_A$ and $\Theta_B$,
$\delta^{\mu,\nu}$ is an isomorphism of $(A,\alpha)-(B,\beta)$ correspondence actions, as desired,
and this completes the verification that the functor $\cpce$ is full.

We now show that $\cpce$ is faithful.
It suffices to show that if $[X,\zeta]\:(A,\delta)\to (B,\epsilon)$ is a morphism in $\coe$, then 
\[
\bigl(\fix(X\rtimes_\zeta G,\what\zeta,j_G^\delta,j_G^\epsilon),
\delta_{X\rtimes_\zeta G}\bigr)
\variso (X,\zeta)
\]
as $(A,\delta)-(B,\epsilon)$ correspondences,
where we regard
$\fix(X\rtimes_\zeta G,\what\zeta,j_G^\delta,j_G^\epsilon)$
as an $A-B$ correspondence
via the isomorphisms
\begin{align*}
\psi_A&\:\fix(A\rtimes_\delta G,\what\delta,j_G^\delta)=A^m\variso A
\\
\psi_B&\:\fix(B\rtimes_\epsilon G,\what\epsilon,j_G^\epsilon)=B^m\variso B.
\end{align*}
Again we use linking algebra techniques.
Let $K=\KK(X)$
and $L=L(X)$,
with associated coactions $\eta$ and $\xi=\smtx{\eta&\zeta\\{*}&\epsilon}$, respectively.
We have the isomorphism
\[
\psi_L\:(L^m,\xi^m)\variso (L,\xi)
\]
given by the maximalization map.
On the other hand,
\begin{align*}
L^m
&=\fix(L\rtimes_\xi G,\what\xi,j_G^\xi)
\\&=\mtx{\fix(K\rtimes_\eta G,\what\eta,j_G^\eta)
&\fix(X\rtimes_\zeta G,\what\zeta,j_G^\zeta)
\\{*}&\fix(B\rtimes_\epsilon G,\what\epsilon,j_G^\epsilon)}
\\&=\mtx{K^m&X^m\\{*}&B^m},
\end{align*}
where the notation $X^m$ is defined by the equations,
and
\begin{align*}
\xi^m
&=C(\wilde\xi)
=\mtx{C(\wilde\eta)&C(\wilde\zeta)
\\{*}&C(\wilde\epsilon)}
\\&=\mtx{\eta^m&\zeta^m\\{*}&\epsilon^m},
\end{align*}
where the notation $\zeta^m$ is defined by the equations.
Since $\psi_L$ preserves the corner projections,
it restricts on the corners to a 
$\kappa^m-\beta^m$ equivariant
Hilbert-bimodule isomorphism
\[
(\psi_K,\psi_X,\psi_B)\:(K^m,X^m,B^m)\variso (K,X,B)
\]
(which includes the definition of the notation $\psi_X$).
We also have the isomorphism
\[
\psi_A\:(A^m,\delta^m)\variso (A,\delta),
\]
and the diagram
\[
\xymatrix{
A^m \ar[r]^-{\psi_A}_-\simeq \ar[d]_{\varphi_A^m}
&A \ar[d]^{\varphi_A}
\\
K^m \ar[r]_-{\psi_K}^-\simeq
&K
}
\]
of morphisms in $\cs$ commutes by
nondegenerate Landstad duality.
Thus, incorporating the isomorphisms $\psi_A$ and $\psi_B$,
$\psi_X$ is an isomorphism of $(A,\delta)-(B,\epsilon)$ correspondence coactions, as desired,
and this completes the verification that the functor $\cpce$ is faithful.

It is clear from the above arguments that $\fixe$ is a quasi-inverse of the equivalence $\cpce$,
indeed it is the unique quasi-inverse with object map
\[
(A,\alpha,\mu)\mapsto \bigl(\fix(A,\alpha,\mu),\delta^\mu\bigr).
\qedhere
\]
\end{proof}

\begin{rem}
Reasoning quite similar to that in the proof of \cite[Remark~6.6]{koqlandstad} shows that the above category equivalence does not give a good inversion,
in the sense of \cite[Definition~4.1]{koqlandstad},
of the crossed-product functor
$\coem\to\cse$
defined on objects by
$(A,\delta)\mapsto A\rtimes_\delta G$.
\end{rem}

\begin{cor}\label{cpa en equiv}
The functor $\cpae\:\eqace\to \kcoe^m$ is a category equivalence.
\end{cor}

\begin{proof}
The proof is parallel to that of \corref{cpa equiv},
except that in \secref{enchilada categories} we have already laid some of the groundwork,
namely in \corref{comcoen} we established a category equivalence $\dste\:\kcoe^m\to \coe^m$.
By \thmref{cpcen thm}
we have an equivalence $\fixe=\dste\circ \cpae$.
It follows that $\cpae$ is also an equivalence.
\end{proof}

\section{Toward outer duality}\label{outer duality}

In this final section, which will be largely speculative,
we formulate a possible approach to \emph{outer duality} for maximal coactions.
To help establish the context, we will first summarize the (partial) outer duality for normal coactions \cite[Subsection~6.3]{koqlandstad}, after which we will describe the additional difficulties for maximal coactions.

\subsection*{Categories}

The \emph{outer category $\coo$ of coactions}
has coactions as objects,
and a \emph{morphism} $(\phi,U)\:(A,\delta)\to (B,\epsilon)$ in $\coo$ consists of an $\epsilon$-cocycle $U$ and a morphism $\phi\:(A,\delta)\to (B,\ad U\circ\epsilon)$ in $\co$.
In \cite{koqlandstad} the coactions were required to be normal, but the arguments we used there to establish the existence of the category $\coo$ work just as well for arbitrary coactions.
We write $\coo^m$ and $\coo^n$ for the full subcategories of maximal and normal coactions, respectively.

The \emph{normal fixed-point category $\eqaco^n$ of equivariant actions}
(denoted by $\waco$ in \cite{koqlandstad})
has equivariant actions as objects,
and a \emph{morphism} $\phi\:(A,\alpha,\mu)\to (B,\beta,\nu)$ in $\eqaco^n$
consists of
a morphism $\phi\:(A,\alpha)\to (B,\beta)$ in $\ac$
such that the canonical extension
$\bar\phi\:M(A)\to M(B)$
restricts to a nondegenerate homomorphism
\[
\bar\phi|\:\fix^n(A,\alpha,\mu)\to M(\fix^n(B,\beta,\nu)).
\]
Here we use the adjective ``normal'' for this category because it involves the normal generalized fixed-point algebras $\fix^n(A,\alpha,\mu)$. In \cite{koqlandstad} we did not need this adjective since we did not use any other kind of fixed-point algebra, but here we will also want to consider an analogue for the maximal generalized fixed-point algebras $\fix(A,\alpha,\mu)$.

Given a morphism $(\phi,U)\:(A,\delta)\to (B,\epsilon)$ in $\coo$,
letting $\zeta=\ad U\circ\epsilon$
we get a morphism
$\phi\:(A,\delta)\to (B,\zeta)$ in $\co$
and a morphism
\[
\Omega_U\:(B\rtimes_\zeta G,\what\zeta)
\to (B\rtimes_\epsilon G,\what\epsilon)
\]
in $\ac$ satisfying $\Omega_U\circ j_B^\zeta=j_B^\epsilon$.
\cite[Theorem~6.12]{koqlandstad} says that the assignments
\begin{align}
(A,\delta)&\mapsto (A\rtimes_\delta G,\what\delta,j_G^\delta)
\label{obj}
\\
(\phi,U)&\mapsto \Omega_U\circ (\phi\rtimes G)
\label{mor}
\end{align}
give a functor from $\coo^n$ to $\eqaco^n$ that is essentially surjective and faithful.
We believe that this functor is in fact a category equivalence, but we cannot prove that it is full because we do not have a fully working version of Pedersen's theorem for outer conjugacy of coactions.

Now we describe the additional difficulties we encounter when we attempt to adapt the above to maximal coactions.
The \emph{maximal fixed-point category $\eqaco$
of equivariant actions} has the same objects as $\eqac$,
and a \emph{morphism} $\phi\:(A,\alpha,\mu)\to (B,\beta,\nu)$ in $\eqaco$ consists of a morphism $\phi\:(A,\alpha)\to (B,\beta)$ in $\ac$ such that
the canonical extension
\[
\bar{\phi\rtimes G}\:M(A\rtimes_\alpha G)\to M(B\rtimes_\beta G)
\]
of the crossed-product homomorphism
restricts to a nondegenerate homomorphism
\[
\bar{\phi\rtimes G}|\:\fix(A,\alpha,\mu)\to M(\fix(B,\beta,\nu)).
\]

We can verify that the above gives a category using arguments parallel to \cite{koqlandstad}:

\begin{lem}\label{compose}
With the above definition of morphism, the category $\eqaco$ is well-defined.
\end{lem}

\begin{proof}
We must check that composition of morphisms is defined.
Once we have done this it will be obvious that composition is associative and that we have identity morphisms.
Suppose that $\psi\:(B,\beta,\nu)\to (C,\gamma,\tau)$ is another morphism,
so that $\bar{\psi\rtimes G}$ restricts to a nondegenerate homomorphism
\[
\bar{\psi\rtimes G}|\:\fix(B,\beta,\nu)\to M(\fix(C,\gamma,\tau)).
\]
The composition of $\phi\rtimes G$ and $\psi\rtimes G$ in $\cs$ is the 
nondegenerate homomorphism
\[
\bar{\phi\rtimes G}\circ (\psi\rtimes G)\:A\rtimes_\alpha G\to M(C\rtimes_\gamma G).
\]
On the other hand, the composition of the nondegenerate homomorphisms
$\bar{\phi\rtimes G}|$ and $\bar{\psi\rtimes G}|$ is the nondegenerate homomorphism
\[
\bar{\bar{\phi\rtimes G}|}\circ \bar{\psi\rtimes G}|\:\fix(A,\alpha,\mu)\to M(\fix(C,\gamma,\tau)).
\]
It is clear from the definitions that this composition is the restriction of
\[
\bar{\bar{\phi\rtimes G}\circ (\psi\rtimes G)}
=\bar{\phi\rtimes G}\circ\bar{\psi\rtimes G}
\]
to $\fix(A,\alpha,\mu)$.
\end{proof}

The \emph{semi-comma equivariant category $\eqacd$ of actions}
has the same objects as $\eqac$,
namely equivariant actions,
and a morphism $\phi\:(A,\alpha,\mu)\to (B,\beta,\nu)$ in the category is just
a morphism $\phi\:(A,\alpha)\to (B,\beta)$ in $\ac$,
i.e., the morphism in $\eqacd$ has nothing to do with $\mu$ and $\nu$.

\begin{prop}\label{dumb functor}
With the above notation,
the assignments
\eqref{obj}--\eqref{mor}
give a functor
$\cpcd\colon \coo\to \eqacd$
that is essentially surjective.
\end{prop}

\begin{proof}
It follows immediately from the definitions that the object and morphism maps \eqref{obj}--\eqref{mor}
are well-defined and
that \eqref{mor}
preserves identity morphisms.
To check that $\cpcd$ preserves compositions,
suppose we are
given morphisms
\[
\xymatrix{
(A,\delta) \ar[r]^-{(\phi,U)} &(B,\epsilon) \ar[r]^-{(\psi,V)} &(C,\zeta)
}
\]
in $\coo$.
The following calculation is dual to one in \cite[toward beginning of proof of Proposition~5.8]{koqlandstad},
and in the case for normal coactions it was omitted from \cite[proof of Theorem~6.12]{koqlandstad},
so we include it here:
\begin{align*}
&\cpcd(\psi,V)\circ \cpcd(\phi,U)\circ j_A
\\&\quad=\bigl(\Omega_V\circ (\psi\rtimes G)\bigr)
\circ
\bigl(\Omega_U\circ (\phi\rtimes G)\bigr)
\circ
j_A
\\&\quad=\Omega_V\circ (\psi\rtimes G)\circ \Omega_U\circ j_B^{\ad U\circ\epsilon}\circ \phi
\\&\quad=\Omega_V\circ (\psi\rtimes G)\circ j_B^\epsilon\circ \phi
\\&\quad=\Omega_V\circ j_C^{\ad V\circ\zeta}\circ \psi\circ \phi
\\&\quad=j_C^\zeta\circ \psi\circ \phi
\\&\quad=\Omega_{(\psi\otimes\id)(U)V}\circ j_C^{\ad (\psi\otimes\id)(U)V}\circ \psi\circ\phi
\\&\quad=\Omega_{(\psi\otimes\id)(U)V}\circ \bigl((\psi\circ\phi)\rtimes G\bigr)\circ j_A
\\&\quad=\cpcd\bigl(\psi\circ\phi,(\psi\otimes\id)(U)V\bigr)\circ j_A
\\&\quad=\cpcd\bigl((\psi,V)\circ (\phi,U)\bigr)\circ j_A,
\end{align*}
and
\begin{align*}
&\bigl(\cpcd(\psi,V)\circ \cpcd(\phi,U)\circ j_G^\delta\otimes\id\bigr)(w_G)
\\&\quad=\Bigl(\cpcd\bigl((\psi,V)\circ (\phi,U)\bigr)\circ j_G^\delta\otimes\id\Bigr)(w_G)
\end{align*}
by a calculation identical to one in the proof of \cite[Theorem~6.12]{koqlandstad},
which implies
\[
\cpcd(\psi,V)\circ \cpcd(\phi,U)\circ j_G^\delta
=\cpcd\bigl((\psi,V)\circ (\phi,U)\bigr)\circ j_G^\delta.
\]
Thus $\cpcd\colon \coo\to \eqacd$ is a functor.

It is clear that $\cpcd$ is essentially surjective, because
it is essentially surjective for the nondegenerate categories, which have the same objects,
and isomorphism in $\eqac$ is stronger than in $\eqacd$.
\end{proof}

Now
the extra difficulties begin:
we want the assignments \eqref{obj}--\eqref{mor} to give a functor $\coo^m\to \eqaco$,
but we have not been able to prove that \eqref{mor} takes morphisms to morphisms.
Of course we do have a morphism
\[
\Omega_U\circ (\phi\rtimes G)\:(A\rtimes_\delta G,\what\delta)\to (B\rtimes_\epsilon G,\what\epsilon)
\]
in $\ac$,
but then we would need to show that
the canonical extension
\[
\bar{\bigl(\Omega_U\circ (\phi\rtimes G)\bigr)\rtimes G}\:
A\rtimes_\delta G\rtimes_{\what\delta} G
\to M(B\rtimes_\epsilon G\rtimes_{\what\epsilon} G)
\]
of
the crossed-product homomorphism
restricts to a nondegenerate homomorphism
\[
\bar{\bigl(\Omega_U\circ (\phi\rtimes G)\bigr)\rtimes G}|\:
\fix(A\rtimes_\delta G,\what\delta,j_G^\delta)
\to M\bigl(\fix(B\rtimes_\epsilon G,\what\epsilon,j_G^\epsilon)\bigr);
\]
this time, unlike with normal generalized fixed-point algebras in \cite{koqlandstad}, the maximal generalized fixed-point algebras seem to depend upon the homomorphisms $j_G$ to a greater degree than we can accommodate.
The homomorphism $\phi\rtimes G$ presents no problem,
because we can apply the functor $\fix\:\eqac\to \co$ to it.
The problem is that $\Omega_U$ does not relate the maps $j_G^\zeta$ and $j_G^\epsilon$,
and the generalized fixed-point algebra
\[
\fix(B\rtimes_\epsilon G,\what\epsilon,j_G^\epsilon)
=C(B\rtimes_\epsilon G\rtimes_{\what\epsilon} G,j_G^\epsilon\rtimes G)
\]
depends explicitly upon $j_G^\epsilon$ by construction.
This is in contrast to the situation for normal coactions,
where the generalized fixed-point algebra coincided with $j_B(B)\subset M(B\rtimes_\epsilon G)$.

Thus, here we have even less than we did in \cite{koqlandstad} ---
not only do we not have a faithful functor,
we do not have a functor at all.
So the problem remains:

\begin{q}
Do the assignments \eqref{obj}--\eqref{mor} give a functor from $\coo^m$ to $\eqaco$?
If not, can the category $\eqaco$ be adjusted so that the assignment
\[
(A,\delta)\mapsto (A\rtimes_\delta G,\what\delta,j_G)
\]
is the object map of a category equivalence
from $\coo^m$ to $\eqaco$?
\end{q}

\begin{rem}
We plan to pursue outer duality for coactions in future work.
One aspect that we will study is the following:
suppose we are given a morphism $(\phi,U)\:(A,\delta)\to (B,\epsilon)$ in the outer category $\coo$ of coactions.
Let $\psi_B\:(B,\epsilon)\to (B^n,\epsilon^n)$ be the normalization.
Then $U^n:=(\psi_B\otimes\id)(U)$ is an $\epsilon^n$-cocycle,
and $(\phi^n,U^n)\:(A^n,\delta^n)\to (B^n,\epsilon^n)$ is a morphism in the subcategory $\coo^n$ of normal coactions.
It should be possible to prove that the functor $\cpc^m\:\coo^m\to \eqacd$ is naturally isomorphic to the composition of
the restricted normalization functor $\nor^m\:\coo^m\to \coo^n$
followed by
$\cpc^n$.
We know from \cite{koqlandstad} that $\cpc^n$ is an equivalence with the category $\eqaco$.
In particular, it should follow that $\cpc^m$ is faithful and/or full if and only if
$\nor^m$
is.
This gives rise to the following questions:
(1)
if $U$ and $V$ are distinct cocycles for a maximal coaction $\delta$,
are the associated $\delta$-cocycles $U^n$ and $V^n$ also distinct,
and
(2)
if $\delta$ and $\epsilon$ are exterior equivalent normal coactions,
are their maximalizations $\delta^m$ and $\epsilon^m$ also exterior equivalent?
\end{rem}


\providecommand{\bysame}{\leavevmode\hbox to3em{\hrulefill}\thinspace}
\providecommand{\MR}{\relax\ifhmode\unskip\space\fi MR }
\providecommand{\MRhref}[2]{%
  \href{http://www.ams.org/mathscinet-getitem?mr=#1}{#2}
}
\providecommand{\href}[2]{#2}

\end{document}